\documentclass[11pt]{article}
\usepackage{amssymb,amsmath,amsfonts,amsthm,epsfig,latexsym,color}
\usepackage{amsfonts,amsmath,amsthm, amssymb, amscd, mathrsfs}
\usepackage{latexsym, euscript, epic, eepic}
\usepackage{graphicx}
\usepackage{verbatim}
\usepackage{fourier}

\usepackage[colorlinks=true,backref=page]{hyperref}

\makeatletter
\newcommand{\subsectionruninhead}{\@startsection{subsection}{2}{0mm}
{-\baselineskip}{-0mm}{\bf\large}}
\newcommand{\subsubsectionruninhead}{\@startsection{subsubsection}{3}{0mm}
{-\baselineskip}{-0mm}{\bf\normalsize}}
\makeatother

\newtheorem*{theorem*}{Theorem}

\newtheorem{theoremalph}{Theorem}

\newtheorem*{proposition*}{Proposition}
\newtheorem*{corollary*}{Corollary}
\newtheorem*{claim*}{Claim}
\newtheorem*{remark*}{Remark}
\newtheorem*{problem*}{Problem}
\newtheorem{theorem}{Theorem}[section]

\newtheorem{proposition}[theorem]{Proposition}
\newtheorem{corollary}[theorem]{Corollary}
\newtheorem{lemma}[theorem]{Lemma}

\newtheorem{claim}[theorem]{Claim}
\theoremstyle{definition}
\newtheorem{definition}[theorem]{Definition}
\newtheorem{remark}[theorem]{Remark}

\numberwithin{equation}{section}

 \def\NN{{\mathbb N}}  
 \def\RR{{\mathbb R}}

   \def\cO{\mathcal{O}} \def\cU{\mathcal{U}}
   \def\cP{\mathcal{P}} 
\def\cE{\mathcal{E}}  \def\cK{\mathcal{K}}

\newcommand{\supp}{\operatorname{\,Supp}}

\newcommand{\orb}{\operatorname{Orb}}

\newcommand{\diam}{\operatorname{Diam}}

\newcommand{\sing}{\operatorname{Sing}}

\newcommand{\eps}{\varepsilon}
\newcommand{\R}{\mathbb{R}}

\newcommand{\htop}{h_{\text{top}}}

\newcommand{\F}{\mathfrak{F}}

\begin{document}

\title{On multifractal analysis and large deviations \\ of singular hyperbolic attractors}

\author{Yi Shi\footnote{Y. Shi was partially supported by National Key R\&D Program of China (2021YFA1001900) and NSFC (12071007, 11831001, 12090015).}\, ,\,  
	Xueting Tian\footnote{X. Tian was partially supported by NSFC (12071082) and Science and Technology Innovation Action Program of Science and Technology Commission of Shanghai Municipality (STCSM, No. 21JC1400700).}  \,   
	Paulo Varandas\footnote{P. Varandas was partially supported by the grant CEECIND/03721/2017 of the Stimulus of Scientific Employment, Individual Support 2017 Call, awarded by FCT-Portugal, and by 
CMUP, which is financed by national funds through FCT-Portugal, under the project UIDB/00144/2020.} \, and  
Xiaodong Wang\footnote{X. Wang was partially supported by National Key R\&D Program of China (2021YFA1001900), NSFC (12071285, 11701366),  Science and Technology Innovation Action Program of Science and Technology Commission of Shanghai Municipality (STCSM, No. 20JC1413200) and Innovation Program of Shanghai Municipal Education Commission (No. 2021-01-07-00-02-E00087).}}

\maketitle

\begin{abstract}
	 In this paper we study the multifractal analysis and large derivations for singular hyperbolic attractors, including the geometric Lorenz attractors. 
	 For each singular hyperbolic homoclinic class whose periodic orbits are all homoclinically related and such that the space of ergodic probability measures 
	 is connected, we prove that: (i) level sets associated to continuous observables are dense in the homoclinic class and satisfy a variational principle; (ii) 
	 irregular sets are either empty or are Baire generic and carry full topological entropy.
	 The assumptions are satisfied by $C^1$-generic singular hyperbolic attractors and $C^r$-generic geometric Lorenz attractors $(r\ge 2)$. 
	Finally we prove level-2 large deviations bounds for weak Gibbs measures, which provide a large deviations principle in the special
	  case of Gibbs measures.
	 The main technique we apply is the horseshoe approximation property. 
\end{abstract}


\section{Introduction}

\qquad Ergodic theorems appear as cornerstones in ergodic theory and dynamical systems, as they allow to describe long time behavior of points in full measure sets 
with respect to invariant probability measures. Given this starting point, a particularly important topic of interest is to characterize level sets, the velocity of convergence to time averages and the set of points for which time averages do not exist, often called irregular points.  In this paper we will be interested
in the multifractal analysis and large deviations  for flows with singularities, whose concepts we will recall.

\vskip2mm
Let $\mathscr{X}^r(M)$, $r\ge 1$, denote the space of $C^r$-vector fields on a closed smooth Riemannian manifold $M$ endowed with the $C^r$-topology.
Given a vector field $X\in\mathcal{X}^1(M)$ and a compact invariant subset $\Lambda$ of  the $C^1$-flow  $(\phi_t)_{t\in\mathbb{R}}$ generated by $X$, we denote by $C(\Lambda,\mathbb{R})$ the space of continuous functions on $\Lambda$.
For any $g\in C(\Lambda,\mathbb{R})$, Birkhoff's ergodic theorem ensures that for any $\mu\in \mathcal{M}_{inv}(\Lambda)$, the time average 
$
\displaystyle\lim_{T\rightarrow\infty}\frac{1}{T}\int_{0}^{T}g(\phi_t(x)) {\rm d}t 
$ 
exists for $\mu$-almost all point $x\in \Lambda$.
Defining, for each $a\in \mathbb{R}$, the \emph{$g$-level set}
$$
R_{g}(a)\colon=\left\{x\in\Lambda \colon \lim_{T\rightarrow\infty}\frac{1}{T}\int_{0}^{T}g(\phi_t(x)) {\rm d}t =a\right\},
$$ 
and the \emph{$g$-irregular set} by
$$
I_{g}\colon=\left\{x\in\Lambda \colon \lim_{T\rightarrow\infty}\frac{1}{T}\int_{0}^{T}g(\phi_t(x)) \,{\rm d}t \text{ does not exist}\right\},
$$ 
one obtains the  \textit{multifractal decomposition} 
$$\Lambda= I_{g} \;\cup\; \bigcup_{a\in\mathbb R} R_{g}(a)$$ 
of the flow with respect to the observable $g$.
The properties of the entropy, dimension or genericity of level sets and irregular sets have been much studied in the recent years.
For uniformly hyperbolic systems, both diffeomorphisms and vector fields, a rigorous mathematical theory of multifractal analysis is available (see ~\cite{BS00b,BPS97,BS00a,DK,Olsen,OlsenWinter,TV99} and references therein). In rough terms, for uniformly hyperbolic dynamical systems 
each level set carries all ergodic information and their topological entropy satisfies a variational principle using the invariant measures supported on it, while
the irregular set carries full topological entropy and, under some conformality assumption, it has full Hausdorff dimension.
The multifractal analysis of non-uniformly hyperbolic systems had a few contributions
(see e.g. 	\cite{BI06,Cli,PT10,TV17}) but the theory still remains quite incomplete, especially in the time-continuous setting.  

\vskip2mm

A second and related topic of interest concerns the theory of large deviations which, in the dynamical systems framework, 
addresses on the rates of convergence in the ergodic theorems. This gives a finer description of the behavior inside the level sets
of the multifractal decomposition described above.  More precisely, given a reference probability measure $\nu$  on
$\Lambda$ (possibly non-invariant) one would like to provide
sharp estimates for the $\nu$-measure the deviation sets
$$
\left
\{
    x \in \Lambda : \frac{1}{T}\int_{0}^{T}g(\phi_t(x)) {\rm d}t  > c
\right\}
$$
for all $g \in C(\Lambda, \mathbb R)$ and all real numbers $c$. 
Large deviations in dynamical systems are often used to measure the velocity of convergence to a certain probability. 
For uniformly hyperbolic and certain non-uniformly hyperbolic systems, both diffeomorphisms and vector fields, large deviation have been well-studied
(see e.g. \cite{Ar07,BV19, CYZ, Ki90,You90} and references therein). From the technical viewpoint, in most situations 
large deviations principles often rely on the following mechanisms: (i) differentiability of the pressure function;
(ii) some gluing orbit property and weak Gibbs estimates; (iii) existence of Young towers with exponential tails modeling the dynamical system; 
or (iv) entropy-denseness on the approximation by horseshoes.
We refer the reader to \cite{CTY,MN,PS05,Va12} for more details on each of these approaches.

\vskip2mm

In this paper we are interested in the multifractal analysis and large deviations of vector fields with singularities, 
including geometric Lorenz attractors (cf. Definition~\ref{Def:lorenz}).
The Lorenz attractor was observed by E. Lorenz~\cite{Lorenz} in 1963, whose dynamics sensitively depend on initial conditions. Later, J. Guckenheimer~\cite{guck} and V. Afra$\breve {\rm \i}$movi$\check{\rm c}$-V. Bykov-L. Sil'nikov~\cite{abs} introduced a geometric model for the 
Lorenz attractor, nowadays known as geometric Lorenz attractors.  It is known that the space of $C^r$ vector fields ($r\in\mathbb{N}_{\geq 2}$) 
exhibiting a geometric Lorenz attractor is an open subset in $\mathscr{X}^r(M^3)$ (cf. \cite{STW}).
In the study of $C^1$ robustly transitive flows, Morales, Pac\'ifico and Pujals ~\cite{mpp} introduced the concept of singular hyperbolic flows
(see Definition \ref{Def:singular-hyp}), which include the geometric Lorenz attractors as a special class of examples. 
The coexistence of singular and regular behavior is known to present difficulties to both the geometric theory and ergodic theory of flows, and to
present new and rich phenomena in comparison to uniformly hyperbolic flows. 
In order to illustrate this, let us mention that for each $r\ge 2$, there exist a Baire residual subset $\mathcal{R}^r$ and a dense subset $\mathcal{D}^r$ in $\mathscr{X}^r(M)$ such that for a geometric Lorenz attractor $\Lambda$ of  $X\in\mathcal{R}^r$, the space of ergodic probability measures supported on $\Lambda$ is connected, while for a geometric Lorenz attractor $\Lambda'$ of  $X\in\mathcal{D}^r$, the space of ergodic probability measures supported on $\Lambda'$  is not connected; 
and a similar statement holds for $C^1$-singular hyperbolic attractors in higher dimension \cite[Theorem A\&B]{STW}. 
In rough terms, the underlying idea is that while the set
$$
\mathcal{M}_1(\Lambda)=\Big\{\mu\in \mathcal{M}_{inv}(\Lambda)\colon \mu(\sing(\Lambda))=0 \Big\}
$$
of probability measures which give zero weight to the singularity set $\sing(\Lambda)$ of a singular hyperbolic attractor $\Lambda$ 
is formed by hyperbolic measures, which inherit a good approximation by periodic orbits, Dirac measures at singularities can be either approximated or not (in the weak$^*$ topology) by periodic measures 
depending on the recurrence of the singular set to itself, measured in terms of proximity to 
vector fields displaying homoclinic loops (see \cite{STW} for the precise statements).

\vskip2mm

Here we use the horseshoe approximation technique to study the multifractal ana\-lysis and large deviations for singular hyperbolic attractors, including geometric Lorenz attractors. 
The classical approach to describe level sets $R_{g}(a)$ involves the uniqueness of equilibrium states for H\"older continuous observables $g$. 
However this is still an open problem for singular hyperbolic attractors. 
On the other hand, if the observable $g$ is merely continuous one knows that uniqueness of equilibrium states fails even for subshifts of finite type
\cite{Hof}.
%
Another obstruction appears when one considers the irregular set $I_{g}$. Indeed, while
most constructions of fractal sets with high entropy involve the use of some specification-like property,
the presence of hyperbolic singularities constitutes an obstruction for specification (see e.g. \cite{SVY15,Wen-Wen20}).
In this direction, a standard argument is to establish the variational principle for saturated sets of generic points. However, up to now it is still unknown whether an invariant  non-ergodic probability measure (for example, convex sum of infinite periodic measures not supported on a same horseshoe) has generic points. 
Similar obstructions occur in the study of large deviations. The drawback of looking 
for the differentiability of the pressure function is that, even in the hyperbolic context, it demands one to consider the space of H\"older continuous observables. The latter relies ultimately on the uniqueness of equilibrium states for H\"older continuous
potentials, a question which in such generality remains widely open. 
Moreover, singular hyperbolic attractors seem not to display any gluing orbit property, as hinted by \cite{BoTV20, SVY15,Wen-Wen20}.
We first show that the horseshoe approximation technique, valid for $C^r$-generic geometric Lorenz attractors and $C^1$-generic singular hyperbolic attractors, 
is enough to show that the level sets and irregular sets of such singular hyperbolic attractors inherit the properties of the corresponding objects for special
classes of horseshoes approximating them. 
This property plays a crucial role in the proof of level-2 large deviations lower bounds for weak Gibbs measures, ie, lower bounds on the measure of points
whose empirical measures belong to some weak$^*$ open set of probability measures on the attractor $\Lambda$. 
On the other hand, level-2 large deviations upper bounds for weak Gibbs measures follow a more standard approach, exploring ideas from the proof of the variational principle. However, as one requires a very mild Gibbs property, the large deviations rate function takes into account certain tails of constants
often associated to the loss of uniform hyperbolicity. We refer the reader to \cite{Va12} for a discussion on the relation between the weak Gibbs property  
and the tail of hyperbolic times 
of local diffeomorphisms, and to Theorem~\ref{Thm:LD2B} for the precise statements.

\subsection*{Organization of the paper}
In Section~\ref{Section:Pre}, we present some  concepts and  known results. 
The main theorems (Theorem A,B,C) of this paper are presented in Section~\ref{Section:statements}.
In Section~\ref{Section:ED-HS-approximation}, we introduce the notions of entropy denseness and horseshoe approximation properties, and provide sufficient conditions for these properties to be verified.
Section~\ref{Section:proof-MA} is devoted to the multifractal analysis of Lorenz attractors/singular hyperbolic attractors and to the proof of Theorems ~\ref{Thm:irregular-Lorenz} and ~\ref{Thm:irregular-SinHyp-attractor}. 
Finally, in Section~\ref{Section:proof-LD} we provide large deviation estimates for Lorenz attractors/singular hyperbolic attractors and prove Theorem~\ref{Thm:LD2B-Lorenz-SH-attractor}.

\section{Preliminaries}\label{Section:Pre}
Before the statements of the main theorems, one would like to give some preliminaries in this section.
Let $\mathscr{X}^r(M)$, $r\ge 1$, denote the space of $C^r$-vector fields on a closed smooth Riemannian manifold $M$. For $X\in\mathscr{X}^r(M)$, denote by $\phi_t^X$ or $\phi_t$ for simplicity the $C^r$-flow generated by $X$ and denote by ${\rm D}\phi_t$ the tangent map of $\phi_t$. Moreover, given any $\phi_t$-invariant set $\Lambda$ we denote by $\sing(\Lambda)$ the set of singularities for the vector field $X$  in $\Lambda$.  
Let $\mathcal M(\Lambda)$ be the space of all probability measures supported on $\Lambda$ endowed with the weak*-topology. Let $d^*$ be a metric on the space $\mathcal{M}(\Lambda)$ compatible with the weak$^*$ topology which can be defined as follows, see for instance~\cite[Section 6.1]{Wa}:
take (and fix) a countable dense subset $\{\varphi_i\}_{i=1}^{\infty}$ of $C(\Lambda,\mathbb{R})$ where  $\varphi_i$ is not the zero function for every $i\geq 1$,
and for any  $\mu,\nu\in \mathcal{M}(\Lambda)$ set
$$\displaystyle d^*(\mu,\nu)=\sum_{i=1}^{\infty}\frac{|\int \varphi_i {\rm d}\mu-\int \varphi_i {\rm d}\nu|}{2^i\|\varphi_i\|},$$
where $\|\cdot\|$ denotes the $C^\infty$-norm on $C(\Lambda,\mathbb{R})$.
The set of invariant (resp. ergodic) probability measures of $X$ supported on $\Lambda$ is denoted by $\mathcal{M}_{inv}(\Lambda)$ (resp. $\mathcal{M}_{erg}(\Lambda)$). 
We denote by $h_\mu(X)$ the metric entropy of the invariant probability measure $\mu\in \mathcal{M}_{inv}(\Lambda)$, defined as the metric entropy of $\mu$ with respect to the time-1 map $\phi^X_1$ of the flow. 

\subsection{Geometric Lorenz attractor and singular hyperbolicity}\label{subsec:lorenz}

We recall the definitions of hyperbolic and singular hyperbolic sets.

\begin{definition}\label{Def:hyp}
	Given a vector field $X\in\mathscr{X}^1(M)$, a compact $\phi_t$-invariant set $\Lambda$ is {\it hyperbolic} if $\Lambda$ admits a continuous ${\rm D}\phi_t$-invariant splitting $T_{\Lambda}M=E^s\oplus \langle X\rangle\oplus E^u$, where $\langle X\rangle$ denotes the one-dimensional linear space generated by the vector field, and $E^s$ (resp. $E^u$) is uniformly contracted (resp. expanded) by ${\rm D}\phi_t$, that is to say, there exist constants $C>0$ and $\eta>0$ such that for any $x\in\Lambda$ and any $t\geq 0$,
	\begin{itemize}
		\item  $\|{\rm D}\phi_t(v)\|\leq Ce^{-\eta t}\|v\|$, for any $v\in E^s(x)$; and
		\item  $\|{\rm D}\phi_{-t}(v)\|\leq Ce^{-\eta t}\|v\|$, for any $v\in E^u(x)$.
	\end{itemize}
	for any $x\in\Lambda$ and $t\geq 0$.
	If $\Lambda$ is transitive, ie it admits a dense orbit, then the dimension $\dim(E^s)$ of the stable subbundle is constant 
	and is called the {\it index} of the hyperbolic splitting.
\end{definition}

The concept of singular hyperbolicity was introduced by Morales-Pacifico-Pujals~\cite{mpp} to describe the geometric structure of Lorenz attractors and these ideas were extended to higher dimensional cases in ~\cite{lgw,Me-Mo}. More explorations can be found in for instance~\cite{Sataev} which establishes SRB measures for singular hyperbolic attractors.  Let us recall this notion.

\begin{definition}\label{Def:singular-hyp}
	Given  a vector field $X\in\mathscr{X}^1(M)$, a compact and invariant set $\Lambda$ is {\it singular hyperbolic} if it admits a continuous ${\rm D}\phi_t$-invariant splitting $T_{\Lambda}M=E^{ss}\oplus E^{cu}$ and constants $C,\eta>0$ such that, 
	for any $x\in\Lambda$ and any $t\geq 0$, 
	\begin{itemize}
		\item $E^{ss}\oplus E^{cu}$ is a {\it dominated splitting} : $\|{\rm D}\phi_t|_{E^{ss}(x)}\|\cdot \|{\rm D}\phi_{-t}|_{E^{cu}(\phi_t(x))}\|< Ce^{-\eta t}$, and
		\item $E^{ss}$ is uniformly contracted by ${\rm D}\phi_t$ : $\|{\rm D}\phi_t(v)\|< Ce^{-\eta t}\|v\|$ for any $v\in E^{ss}(x)\setminus\{0\}$;
		\item $E^{cu}$ is {\it sectionally expanded} by ${\rm D}\phi_t$ :  $|\det{\rm D}\phi_t|_{V_x}| > Ce^{\eta t}$ for any 2-dimensional  subspace $V_x\subset E^{cu}_x$.
	\end{itemize}
\end{definition}

\begin{remark}\label{Rem:singular hyperbolicity}
	The following properties hold:
	
	\vspace{-.3cm}
	\begin{enumerate}
		\item Given $X\in\mathscr{X}^1(M)$, it follows from the definition of singular-hyperbolicity that 
		all hyperbolic periodic orbits of a  singular-hyperbolic set $\Lambda$ have the same index. 
		\vspace{-.15cm}
		\item Given $X\in\mathscr{X}^1(M)$ and an  invariant compact set $\Lambda$ which contains regular points (i.e. points that are not singularities),
		if  $\Lambda$ is hyperbolic, then  it must contain no singularities. On the other hand, if $\sing(\Lambda)=\emptyset$, then $\Lambda$ is hyperbolic if and only if $\Lambda$ is singular hyperbolic for $X$ or for $-X$ (cf. ~\cite{mpp}).
		\vspace{-.15cm}
		\item Singular hyperbolicity is a $C^1$-robust property. More precisely, if $\Lambda$ is a singular hyperbolic invariant compact set of $X\in\mathscr{X}^1(M)$ associated with splitting $T_{\Lambda}M=E^{ss}\oplus E^{cu}$ and constants $(C,\eta)$, then there exists an open neighborhood $U$ of $\Lambda$ and a neighborhood $\mathcal{U}\subset \mathscr{X}^1(M)$ of $X$ such that for any $Y\in\mathcal{U}$, the maximal invariant set of $\phi_t^Y$ in $U$ is a singular hyperbolic set for $Y$ with the same stable dimension and constants $(C,\eta)$ (cf.~\cite[Proposition 1]{MPP-SH-robust}).
	\end{enumerate}
\end{remark}

We give the definition of homoclinic class of a hyperbolic periodic orbit and homoclinically related property between hyperbolic periodic orbits. First recall that for a hyperbolic periodic point $p$ of $X$, the stable/unstable manifolds associated to its orbit $\orb(p)$ is defined as follows:
$$W^s(\orb(p))=\left\{x\colon \lim_{t\rightarrow +\infty}d\big(\phi_t(x),\orb(p)\big)= 0  \right\},$$
$$W^u(\orb(p))=\left\{x\colon \lim_{t\rightarrow +\infty}d\big(\phi_{-t}(x),\orb(p)\big)= 0  \right\}.$$
The classical theory on invariant manifolds by~\cite{hps} ensures that $W^s(\orb(p))$ and $W^u(\orb(p))$ are submanifolds of $M$.

\begin{definition}\label{Def:homoclinic-class}
	Let $X\in\mathscr{X}^1(M)$ and assume $p,q$ are two hyperbolic periodic points of $X$.
	The {\it homoclinic class} of $p$ is defined as
	$$H(p)=\overline{W^{s}(\orb(p)) \pitchfork W^{u}(\orb(p))},$$ 
	that is, the closure of the set of transversal intersections between the stable and  unstable manifolds of the periodic orbit $\orb(p)$ of $p$.  
	The two hyperbolic periodic orbits $\orb(p)$ and $\orb(q)$ are {\it homoclinically related} if the stable manifold  of $\orb(p)$ has non-empty transverse intersections with the unstable manifold of $\orb(q)$ and vice versa. 
	A homoclinic class is  called \emph{non-trivial} if it is not reduced to a single hyperbolic periodic orbit.
\end{definition}

Finally, we give the definition of geometric Lorenz attractors following Guckenheimer and Williams~\cite{guck,gw,williams} for vector fields on a closed 3-manifold $M^3$. 
Recall that for a hyperbolic singularity $\sigma$ of a vector field $X\in\mathscr{X}^1(M)$, its local  stable manifold of size $\delta$ is 
$$W_{\delta}^s(\sigma)=\left\{x\colon \lim_{t\rightarrow+\infty} \phi_t(x)=\sigma \text{~and~} d(\phi_t(x),\sigma)<\delta,\forall t\geq 0  \right\}.$$
When we do not specify the size $\delta$ of the local stable manifold, we just write as $W^s_{loc}(\sigma)$.

\begin{definition}\label{Def:lorenz}
	We say $X\in\mathscr{X}^r(M^3)$ ($r\geq1$) exhibits a \emph{geometric Lorenz attractor} $\Lambda$, if $X$ has an attracting region $U\subset M^3$ such that  $\Lambda=\bigcap_{t>0}\phi^X_t(U)$ is a singular hyperbolic attractor and satisfies:
	\begin{itemize}
		\item There exists a unique singularity $\sigma\in\Lambda$ with three exponents $\lambda_1<\lambda_2<0<\lambda_3$, which satisfy $\lambda_1+\lambda_3<0$ and $\lambda_2+\lambda_3>0$.  
		\item $\Lambda$ admits a $C^r$-smooth cross section $\Sigma$ which is $C^1$-diffeomorphic to $[-1,1]\times[-1,1]$, such that $l=\{0\}\times[-1,1]=W^s_{\it loc}(\sigma)\cap\Sigma$, and
		for every $z\in U\setminus W^s_{\it loc}(\sigma)$, there exists $t>0$ such that $\phi_t^X(z)\in\Sigma$.
		\item Up to the previous identification, the Poincar\'e map $P:\Sigma\setminus l\rightarrow\Sigma$ is a skew-product map
		$$
		P(x,y)=\big( f(x)~,~H(x,y) \big), \qquad \forall(x,y)\in[-1,1]^2\setminus l.
		$$
		Moreover, it satisfies
		\begin{itemize}
			\item $H(x,y)<0$ for $x>0$, and $H(x,y)>0$ for $x<0$; \color{black} 
			\item $\sup_{(x,y)\in\Sigma\setminus l}\big|\partial H(x,y)/\partial y\big|<1$, 
			and
			$\sup_{(x,y)\in\Sigma\setminus l}\big|\partial H(x,y)/\partial x\big|<1$; \color{black} 
			\item the one-dimensional quotient map $f:[-1,1]\setminus\{0\}\rightarrow[-1,1]$ is $C^1$-smooth and satisfies
			$\lim_{x\rightarrow0^-}f(x)=1$,
			$\lim_{x\rightarrow0^+}f(x)=-1$, $-1<f(x)<1$ and
			$f'(x)>\sqrt{2}$ for every $x\in[-1,1]\setminus\{0\}$.
		\end{itemize}
	\end{itemize}
\end{definition}

\begin{figure}[htbp]
	\centering
	\includegraphics[width=14cm]{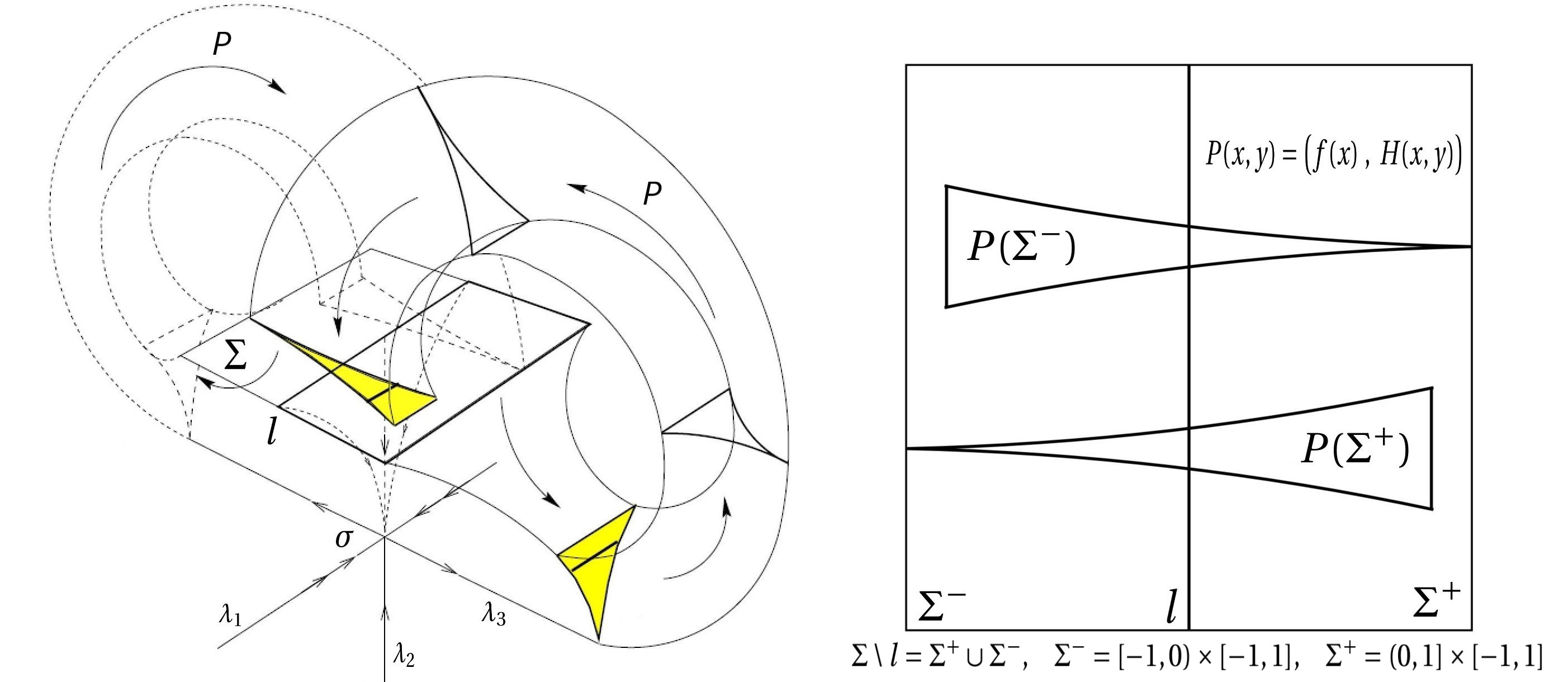}
	\caption{Geometric Lorenz attractor and return map}
\end{figure}

It has been proved that geometric Lorenz attractor is a homoclinic class~\cite[Theorem 6.8]{AP}. Moreover, for $r\in\mathbb{N}_{\geq2}\cup\{\infty\}$, vector fields exhibiting a geometric Lorenz attractor forms a $C^r$-open set in $\mathscr{X}^r(M)$~\cite[Proposition 4.7]{STW} .

\begin{proposition}\label{Pro:lorenz-robust}
	Let $r\in\mathbb{N}_{\geq2}\cup\{\infty\}$  and $X\in\mathscr{X}^r(M^3)$. If $X$ exhibits a geometric Lorenz attractor $\Lambda$ with attracting region $U$, then $\Lambda$ is a singular hyperbolic homoclinic class, and every pair of periodic orbits are homoclinic related.
	Moreover, there exists a $C^r$-neighborhood $\cU$ of $X$ in $\mathscr{X}^r(M^3)$, such that for every $Y\in\cU$, $U$ is an attracting region of $Y$, and the maximal invariant set $\Lambda_Y=\bigcap_{t>0}\phi_t^Y(U)$ is a geometric Lorenz attractor. 
	%
\end{proposition}

\subsection{Topological entropy on general subsets}\label{Section:entropy}

Let $(K,d)$ be a compact metric space. 
The topological entropy of a continuous flow $\Phi=(\phi_t)_t$ on a compact invariant set $Z\subset K$ can be defined as the topological entropy of its time-1 map $\phi_1$ on $Z$. 
When $Z$ fails to be invariant or compact, the topological entropy can be defined using Carath\'eodory structures as in the discrete-time framework, see~\cite{Pesin}, also see~\cite[Section 2.1]{Thompson2009}. 
Let us be more precise.
Fix an arbitrary set $Z\subset K$.
For any $x\in K$, for any $\eps>0$ and any $t\geq0$, consider the \textit{$(t,\eps)$-Bowen ball}
$$B(x,t,\eps,\Phi) \colon= \big\{y\in K:d(\phi_{\tau}(x),\phi_{\tau}(y))<\eps,\tau\in[0,t)\big\}.$$
Now, given a finite or countable collections  
$\gamma=\{B(x_i,t_i,\eps,\Phi)\}_i$ of Bowen balls which covers $Z$ and a fixed $s\in \R$, we define:
\begin{eqnarray*}
	Q(Z,s,\gamma,\varepsilon) &=& \sum_{B(x_i,t_i,\eps,\Phi)\in\gamma}\exp\left(-st_i\right).
\end{eqnarray*}
For $N\in\NN$, define
\begin{eqnarray*}
	M(Z,s,\eps,N) &=& \inf_{\gamma} \;Q(Z,s,\gamma,\varepsilon),
\end{eqnarray*}
where the infimum is taken over all finite or countable collections 
$\gamma=\{B(x_i,t_i,\eps,\Phi)\}_i$ of Bowen balls which covers $Z$ such that $x_i\in Z$ and $t_i\geq N$, for all $i\geq 1$. 
Since $M(Z,s,\eps,N)$ is non-decreasing with respect to $N$, the limit
$$
M(Z,s,\eps)\colon=\lim_{N\rightarrow\infty}M(Z,s,\eps,N)
$$
does exist. It is not hard to check that 
$$
\htop(\Phi,Z,\eps) \colon= \inf \big\{s\colon M(Z,s,\eps)=0\big\}=\sup\big\{s\colon M(Z,s,\eps)=\infty\big\}
$$
is well defined. 
Then the \emph{topological entropy} of $Z$ with respect to the flow $\Phi=(\phi_t)$ is
$$\htop(\Phi,Z)\colon=\lim_{\eps\to0}\htop(\Phi,Z,\eps).$$
Note that similarly with the definition of the metric entropy with respect to a flow, the topological entropy of $Z$ with respect to the flow $\Phi=(\phi_t)_t$ equals to that with respect to its time-one map $\phi_1$.

\subsection{Topological pressure and weak Gibbs measure}
In this subsection, we always assume $X\in\mathcal{X}^1(M)$ and $\Lambda$ is an invariant compact set. Recall that we denote by $\Phi=(\phi_t)_{t}$ the flow generated by $X$. One first recalls the notion of topological pressure associated to a potential and equilibrium state.

\begin{definition}\label{Def:pressure}
	Given a continuous function $\psi\colon\Lambda\rightarrow \mathbb{R}$, here called a {\it potential}, 
	the {\it topological pressure} $P_{\rm top}(\Lambda,\psi)$ is defined  as
	$$P_{\rm top}(\Lambda,\psi)=\sup_{\mu\in \mathcal{M}_{inv}(\Lambda)}\left\{h_{\mu}(X)+\int\psi {\rm d}\mu\right\}.$$
	An invariant probability measure attaining the previous supremum is called an {\it equilibrium state} of $(\phi^X_t)_t$ with respect to $\psi$.
\end{definition}


Now we define the notion of {\it weak Gibbs measure} associated to which we will study the large deviations bounds.
\begin{definition}\label{Def:weak-Gibbs}
	Let $X\in\mathscr{X}^1(M)$ and $\Lambda$ be an invariant compact set. Given a H\"older continuous potential $\psi\colon \Lambda\to\mathbb R$, an invariant measure $\mu_\psi\in\mathcal{M}_{inv}(\Lambda)$ is called a \emph{weak Gibbs measure} with respect to  $\psi$ if there exists a full $\mu_\psi$-measure subset $\Lambda_H \subset \Lambda$ and $\eps_0>0$ so that the following holds: for any $x\in \Lambda_{H}$, $t>0$ and $\varepsilon\in(0,\varepsilon_0)$, there exists a constant $C_t(x,\eps)>1$ such that $\lim\limits_{t\to\infty}\frac1t\log C_t(x,\eps)=0$ and
	\begin{equation}\label{Def:Gibbs-weak}
		\frac1{C_t(x,\eps)} \le \frac{\mu_\psi \Big( B\big(y,t,\eps,\Phi\big)  \Big)}{e^{- t\,P_{\rm top}(X,\psi) + \int_0^t \psi(\phi_s(x))\, {\rm d}s}}
		\le  C_t(x,\eps).
	\end{equation}
	for any dynamic Bowen ball $B\big(y,t,\eps,\Phi\big) \subset B\big(x,t,\eps_0,\Phi\big)$.
\end{definition}

Such a weak Gibbs property in Definition~\ref{Def:weak-Gibbs} holds for large classes of non-uniformly hyperbolic dynamical systems and conformal probability measures with H\"older continuous Jacobians, in which case $\Lambda_H$ may be chosen as the set of points with infinitely many instants of hyperbolicity (see e.g. \cite{Va12} for more details). 

\begin{remark}\label{Rem:Gibbs}
	Consider the extension $C_t: \Lambda \to [1,+\infty]$ defined by $C_t(x,\eps)=+\infty$ for every $x\in \Lambda\setminus \Lambda_H$.
	It is a standing assumption in the sense that, for each $t>0$, the map $C_t: \Lambda \to [1,+\infty)$ is lower semi-continuous, i.e. for each $a \ge 1$ the set
	$$
	\Big\{x\in \Lambda \colon  C_t(x,\eps) \in [1,a] \Big\} \quad\text{is closed}.
	$$
	A special case of weak Gibbs property occurs when $C_t(x,\eps)$ can be chosen as a constant $C_t(\eps)$ independent of $x$.
	In case the constants $C_t(x,\eps)$ can be chosen independently of both $t$ and $x$ and $\Lambda_H=\Lambda$ the probability measure $\mu$ is called a \emph{Gibbs measure}.
\end{remark}

There is a strong connection between equilibrium state and the notion of  Gibbs measures in the context of uniformly hyperbolic dynamical systems: the space of equilibrium states coincides with the space of  Gibbs measures. Most importantly, the quantitative description of dynamic balls make Gibbs measures extremely useful to compute the speed of convergence in the ergodic theorem (see e.g. \cite{Ar07,BV19,CYZ,You90}).

\subsection{Suspension flows}\label{subsec:suspension}

Let $f\colon K\to K$ be a homeomorphism on a compact metric space $(K,d)$ and consider a continuous roof function $\rho\colon K\to(0,+\infty)$.
We define the suspension space to be
$$K_{\rho}=\{(x,s)\in K\times [0,+\infty) \colon 0\leq s\leq \rho(x)\}/\sim,$$
where the equivalence relation $\sim$ identifies $(x,\rho(x))$ with $(f(x),0)$, for all $x\in K$. 
Let
$\pi$
denote the quotient map from $K\times[0,+\infty)$ to $K_{\rho}$. 
We define the flow
$\F=\{f_t\}$ on the quotient space $K_{\rho}$ by
$$f_t(x,s)=\pi(x,s+t).$$
For any function $g\colon K_{\rho}\to \R$, we associate the function $\varphi_g\colon K\to\R$ by $\displaystyle\varphi_g(x)=\int_0^{\rho(x)}g(x,t){\rm d}t$. Since the roof function $\rho$ is continuous, $\varphi_g$ is continuous as long as $g$ is. Moreover, to each invariant probability $\mu\in \mathcal{M}(f,X)$ we associate the measure $\mu_{\rho}$ given by
$$\displaystyle\int_{K_{\rho}}g {\rm d}\mu_{\rho}=\frac{\int_K\varphi_g {\rm d}\mu}{\int_K\rho {\rm d}\mu} \qquad \forall g\in C(K_{\rho},\R). $$
Observe that since $\rho$ is bounded away from zero, the measure $\mu_{\rho}$ is well-defined and $\F$-invariant, i.e. $\mu_{\rho}(f_t^{-1}A)=\mu_{\rho}(A)$ for all $t\geq0$ and measurable sets $A$. Moreover, the map 
$$
\mathcal{R}\colon\mathcal{M}(f,K)\to \mathcal{M}(\F,K_{\rho}) \quad\text{given by} \quad \mu\mapsto\mu_{\rho}
$$ 
is a bijection. Abramov's theorem \cite{Abramov,Parry-Pollicott} states that 
$\displaystyle h_{\mu_{\rho}}(\F)=h_{\mu}(f)/\int\rho {\rm d }\mu$ and hence,
the topological entropy $\htop(\F)$ of the flow satisfies
$$
\htop(\F)=\sup\Big\{h_{\mu_{\rho}}(\F)\colon\mu_{\rho}\in M(\F,K_{\rho})\Big\}=\sup\left\{\frac{h_{\mu}(f)}{\int\rho {\rm d}\mu}\colon\mu\in M(f,K)\right\}.
$$
Throughout we will use the notation $\Phi=(\phi_t)_{t}$ for a flow on a compact metric space and  
$\F=(f_t)_{t}$ for a suspension flow.
Suspension flows are endowed with a natural metric, known as the Bowen-Walters metric (see e.g. \cite{BS00b}).

\subsection{Specification properties}
For completeness of the paper, we recall the notions of specification property introduced in~\cite{Sigmund70} and the gluing orbit property.

\begin{definition}\label{Def:specification}
	Let $K$ be a compact metric space and $f\colon K\rightarrow K$ be a continuous map. We say $f$ satisfies the {\it specification property} if for any $\varepsilon>0$, there exists an integer $m(\varepsilon)$ such that the following holds: for any finite collection $\{[a_i,b_i]\}_{i=1}^k$ of intervals with $a_i,b_i\in\NN$ and $a_{i+1}-b_i\geq m(\varepsilon), i=1,2,\cdots,k-1$ and for any points $x_1,x_2,\cdots,x_k\in K$, there exists a point $x\in K$ such that 
	$$d(f^{a_i+t}(x),f^t(x_i))<\varepsilon,~~\text{for all $t=1,2,\cdots, b_i-a_i$ and all $i=1,2,\cdots k$}.$$
\end{definition}

The gluing orbit property was introduced in~\cite{BoTV20,BV19,Tian-Sun} (which was called transitive specification property in~\cite{Tian-Sun})  and is a weak version of specification property (see explanations for instance in~\cite[Page 553]{Tian-Sun}).

\begin{definition}\label{Def:g-o-property}
	Let $K$ be a compact metric space and $f\colon K\rightarrow K$ be a continuous map. We say $f$ satisfies the {\it gluing orbit property} if for any $\varepsilon>0$, there exists an integer $M(\varepsilon)$ such that the following holds: for any finite collections of points $\{x_i\}_{i=1}^k$ and integers $\{n_i\}_{i=1}^k$, there exist integers $m_1,m_2\cdots,m_{k-1}\leq M(\varepsilon)$ and a point $x\in K$ such that 
	$$d\left(f^{t+\sum_{j=0}^{i-1}(m_j+n_j)}(x),f^t(x_i)\right)<\varepsilon,~~\text{for all $t=1,2,\cdots, n_i-1$ and all $i=1,2,\cdots k$},$$
	where we set $m_0=n_0=0$.
\end{definition}
We refer the readers to~\cite{BoTV20,BV19,Tian-Sun,Thomp} for more studies of systems admitting specification property or the gluing orbit property.

\section{Statements of main theorems}\label{Section:statements}
\subsection{Multifractal analysis}\label{Section:MA}

Recall that for a $C^r(r\geq 1)$-vector field $X\in\mathscr{X}^r(M)$, we denote by $\phi_t^X\colon M\rightarrow M$ the $C^r$-flow generated by $X$ and by ${\rm D}\phi_t^X$ the tangent map of $\phi_t^X$. We also use $\Phi=(\phi_t)_t$ and ${\rm D}\phi_t$ for simplicity if there is no confusion.
In the following, by $M^3$ we mean a $3$-dimensional closed smooth Riemannian manifold. We use the symbol ``$h_{\rm top}(\cdot)$'' to denote the topological entropy of a set, see detailed definitions in Section~\ref{Section:entropy}. Recall that a subset $\mathcal{R}$ of a topological space $X$ is {\it residual} if it contains a dense $G_\delta$ subset of $X$.
Our first main result ensures that the Birkhoff irregular points form a residual subset of geometric Lorenz attractors, and that level sets are typically dense and satisfy
a relative variational principle. More precisely:

\begin{theoremalph}\label{Thm:irregular-Lorenz}
	There exists a Baire residual subset $\mathcal{R}^r\subset\mathscr{X}^r(M^3),~(r\in\mathbb{N}_{\geq 2})$ so that, 
	if $\Lambda$ is a geometric Lorenz attractor of $X\in\mathcal{R}^r$ and $g\in C(\Lambda,\mathbb{R})$, then either:
		\begin{enumerate}
		\item $I_{g}$ is empty and $\displaystyle\int g {\rm d}\mu=\int g {\rm d}\nu$ for all $\mu,\nu\in\mathcal{M}_{inv}(\Lambda)$, or
		\item  $I_{g}$ is a residual subset of $\Lambda$ and  $h_{\rm top}(I_{g})=h_{\rm top}(\Lambda)$. 
	\end{enumerate}  
	Moreover, if $I_{g}$ is non-empty then for any $a\in \mathbb{R}$ satisfying  
		$$
		\displaystyle\inf_{\mu\in \mathcal{M}_{inv}(\Lambda)}\int g {\rm d}\mu<a<\sup_{\mu\in \mathcal{M}_{inv}(\Lambda)} \int g {\rm d}\mu,
		$$  
		the level set $R_{g}(a)$ is dense in $\Lambda$ and 
		$$h_{\rm top}(R_{g}(a))=\sup\left\{h_\mu(X) \colon \int g {\rm d}\mu=a, \mu\in   \mathcal{M}_{inv}(\Lambda)\right\}.$$ 
\end{theoremalph}
\begin{remark}\label{Rem:Lorenz}
	We consider the original geometric model introduced by Guckenheimer and Williams~\cite{guck,gw} where it is required that the induced one-dimensional Lorenz map $f$ satisfies that $f'>\sqrt{2}$, see Definition~\ref{Def:lorenz}. This is because we have to use the eventually onto property of $f$ to obtain that all periodic orbits in the geometric Lorenz attractor are homoclinically related (Definition~\ref{Def:homoclinic-class}), see Proposition~\ref{Pro:lorenz-robust} and Corollary~\ref{Pro:HS-approximation-generic}.
\end{remark}

For singular hyperbolic attractors, one obtains the following result for $C^1$-generic vector fields.

\begin{theoremalph}\label{Thm:irregular-SinHyp-attractor}
	There exists a Baire residual subset $\mathcal{R}\subset\mathscr{X}^1(M)$ so that
	if  $\Lambda$ is  a singular hyperbolic attractor of $X\in\mathcal{R}$ and $g\in C(\Lambda,\mathbb{R})$ then either
	\begin{enumerate}
		\item $I_{g}$ is empty and and $\displaystyle\int g {\rm d}\mu=\int g {\rm d}\nu$ for all $\mu,\nu\in\mathcal{M}_{inv}(\Lambda)$, or
		\item $I_{g}$ is a residual subset of $\Lambda$ and  $h_{\rm top}(I_{g})=h_{\rm top}(\Lambda)$. 
	\end{enumerate}  
			Moreover, if $I_{g}$ is non-empty then, for any $a\in \mathbb{R}$ satisfying  
		$$
		\displaystyle\inf_{\mu\in \mathcal{M}_{inv}(\Lambda)}\int g {\rm d}\mu<a<\sup_{\mu\in \mathcal{M}_{inv}(\Lambda)} \int g {\rm d}\mu,
		$$  
		the level set $R_{g}(a)$ is dense in $\Lambda$ and 
		$$h_{\rm top}(R_{g}(a))=\sup\left\{h_\mu(X) \colon \int g {\rm d}\mu=a, \mu\in   \mathcal{M}_{inv}(\Lambda)\right\}.$$
\end{theoremalph}

The proofs of Theorem~\ref{Thm:irregular-Lorenz} and~\ref{Thm:irregular-SinHyp-attractor} will be given in Section~\ref{sec:proofsAB} after the proof of a technical theorem--Theorem~\ref{Thm:irregular-levelset}. 

\begin{remark}\label{Rem:Thm-A-B}
	We give two remarks here.
	\begin{enumerate}
	 \item It is clear from Theorem~\ref{Thm:irregular-Lorenz} that, for typical geometric Lorenz attractors, $I_g$ is non-empty if and only if there exist $\mu,\nu\in\mathcal{M}_{inv}(\Lambda)$
	so that $\int g {\rm d}\mu\neq\int g {\rm d}\nu$, a property satisfied by a $C^0$-open and dense set of continuous observables.
	In particular, if $\mathcal{R}^r\subset\mathscr{X}^r(M^3),~(r\in\mathbb{N}_{\geq 2})$ is the Baire residual subset given by 
	Theorem~\ref{Thm:irregular-Lorenz} and $\Lambda$ is a geometric Lorenz attractor of $X\in\mathcal{R}^r$ then
	$$
	\Big\{g\in C(\Lambda,\mathbb{R})\colon \text{$I_{g}$ is residual in $\Lambda$ and  $h_{\rm top}(I_{g})=h_{\rm top}(\Lambda)$} \Big\}
	$$
	is open and dense in $C(\Lambda,\mathbb{R})$. A similar conclusion holds in the context of Theorem~\ref{Thm:irregular-SinHyp-attractor}.
	
	\item For the characterizations of level sets in Theorem~\ref{Thm:irregular-SinHyp-attractor} \& Theorem~\ref{Thm:irregular-Lorenz}, when 
	$$a=\displaystyle\inf_{\mu\in \mathcal{M}_{inv}(\Lambda)}\int g {\rm d}\mu  \text{~~~or~~~}  a=\displaystyle\sup_{\mu\in \mathcal{M}_{inv}(\Lambda)} \int g {\rm d}\mu,$$
	by~\cite[Theorem 1.1]{FH10}, one also has the variational principle that 
	$$h_{\rm top}(R_{g}(a))=\sup\left\{h_\mu(X) \colon \int g {\rm d}\mu=a, \mu\in   \mathcal{M}_{inv}(\Lambda)\right\}.$$
	\end{enumerate}
\end{remark}

\subsection{Large deviations}\label{Section:LD}

Recall that $\mathcal M(\Lambda)$ denotes the space of all probability measures supported on $\Lambda$ endowed with the weak*-topology.
Our next result is a level-2 large deviations principle, which detects exponential convergence to equilibrium on the space $\mathcal M(\Lambda)$ for a singular hyperbolic attractor or a Lorenz attractor $\Lambda$. 
We need to set further notations. 
Given $t>0$, let $\cE_t\colon \Lambda \to \mathcal M(\Lambda)$ be the \emph{empirical measure function at time t}, defined by
$$
\cE_t(x)\colon=\frac1t\int_0^t \delta_{\phi^X_s(x)} \, {\rm d}s.
$$
In other words, $\cE_t(x)$ is the {empirical probability on $\Lambda$} determined by the point $x\in \Lambda$ at time $t$.
We say that an invariant probability measure $\mu$ satisfies a \emph{level-2 large deviations principle} if there exists a lower-semicontinuous rate function $\mathfrak I \colon \mathcal M(\Lambda) \to [0,+\infty]$ so that 
\begin{align*}
\limsup_{t\to +\infty} \frac1t \log \mu \Big(\Big\{ x\in \Lambda \colon \cE_t(x)\in \cK\Big\} \Big)
    \le  -\inf_{\nu \in \cK} \mathfrak I(\nu)
\end{align*}
for any closed subset $\cK\subset \mathcal M(\Lambda)$, and
\begin{align*}
\liminf_{t\to +\infty} \frac1t \log \mu \Big( \Big\{x\in \Lambda \colon \cE_t(x)\in \cO\Big\} \Big)
    \ge  -\inf_{\nu \in \cO} \mathfrak I(\nu)
\end{align*}
for any open subset $\cO \subset \mathcal M(\Lambda)$. If the latter holds then the contraction principle (see e.g. \cite[Theorem~4.2.1]{DZ}) ensures the
level-1 large deviations principle
\begin{align*}
\limsup_{t\to +\infty} \frac1t \log \mu \Big(\Big\{ x\in \Lambda \colon \frac1t\int_0^t g(\phi^X_s(x)) \, {\rm d}s  \in [a,b] \Big\} \Big)
    \le  -\inf \Big\{ \mathfrak I(\nu) \colon \nu\in \mathcal M(\Lambda) \,\&\, \int g\, d\nu \in [a,b]\Big\}
\end{align*}
and
\begin{align*}
\liminf_{t\to +\infty} \frac1t \log \mu \Big(\Big\{ x\in \Lambda \colon \frac1t\int_0^t g(\phi^X_s(x)) \, {\rm d}s  \in (a,b) \Big\} \Big)
    \ge   -\inf \Big\{ \mathfrak I(\nu) \colon \nu\in \mathcal M(\Lambda) \,\&\, \int g\, d\nu \in (a,b) \Big\}
\end{align*}
for each continuous observable $g: \Lambda \to \mathbb R$.

Our next result establishes large deviations bounds for weak Gibbs measures, in which the rate function is expressed 
in terms of thermodynamic quantities. 
For a continuous potential $\psi\colon \Lambda\to\mathbb R$, define 
$\mathfrak I_\psi\colon \mathcal M(\Lambda) \to [0,+\infty]$ by
$$
\mathfrak I_\psi(\mu) 
=
\begin{cases}
	\begin{array}{cl}
		\displaystyle P_{\rm top}(\Lambda,\psi) - h_\mu(X) - \int {\psi} \, {\rm d}\mu & ,\text{if } \mu\in \mathcal M_{inv}(\Lambda); \\
		+\infty & ,\text{otherwise. } 
	\end{array}
\end{cases}
$$
When $\Lambda$ is singular hyperbolic, the entropy map  
$h\colon\mathcal{M}_{inv}(\Lambda)\rightarrow\mathbb{R}, ~~~\mu\mapsto h_{\mu}(X)$
is upper semi-continuous~\cite{PYY}, therefore $\mathfrak I_\psi$ is lower semi-continuous.

\begin{theoremalph}\label{Thm:LD2B-Lorenz-SH-attractor} (Level-2 large deviations) 
	There exist a Baire residual subset $\mathcal{R}^r\subset\mathcal{X}^r(M^3),~(r\in\mathbb{N}_{\geq 2})$ and a Baire residual subset $\mathcal{R}\subset\mathcal{X}^1(M)$ such that if $\Lambda$ is a Lorenz attractor of $X\in \mathcal{R}^r$ or $\Lambda$ is a singular hyperbolic attractor of $X\in\mathcal{R}$, then the following properties are satisfied.\\
	Assume $\mu_\psi$ is a weak Gibbs measure with respect to a H\"older continuous potential $\psi\colon \Lambda\to\mathbb R$ with $\Lambda_H$ being the $\mu_\psi$-full measure set such that \eqref{Def:Gibbs-weak} satisfies. 
	Then one has:
	\begin{enumerate}
		\item\label{item:LD-upper} (upper bound) There exists $c_\infty\leq 0$ so that 
		\begin{align*}
			\limsup_{t\to\infty} & \frac1t \log \mu_\psi \big(\{x\in \Lambda \colon \cE_t(x) \in \mathcal K\} \big)  
			\le 
			\max\Big\{ 
			-\inf_{\mu \in \cK} \mathfrak I_{\psi}(\mu) \;,\;  c_\infty
			\Big\}
		\end{align*}
		for any closed subset $\cK\subset \mathcal M(\Lambda)$.
		
		\item\label{item:LD-lower} (lower bound)    If $\cO \subset \mathcal M(\Lambda)$ is an open set, 
		then
		\begin{align*}
			\liminf_{t\to +\infty} \frac1t \log \mu_\psi \Big( \Big\{x\in \Lambda \colon \cE_t(x)\in \cO \Big\} \Big)
			\ge  \displaystyle -\inf~\big\{\mathfrak I_{\psi}(\nu)\colon \nu\in \cO\cap\mathcal{M}_{erg}(\Lambda), \nu(\Lambda_H)=1\big\}.
		\end{align*}
		
		\item\label{item:LD-lower-Gibbs} (lower bound for Gibbs measure) If $\mu_\psi$ is a Gibbs measure with respect to $\psi$, then
		\begin{align*}
			\liminf_{t\to +\infty} \frac1t \log \mu_\psi \Big( \Big\{x\in \Lambda \colon \cE_t(x)\in \cO\Big\} \Big)
			\ge  -\inf_{\mu \in \cO} \mathfrak I_{\psi}(\mu)
		\end{align*}
		for any open subset $\cO \subset \mathcal M(\Lambda)$.
		
	\end{enumerate}

\end{theoremalph}


Some comments are in order. 
Theorem~\ref{Thm:LD2B-Lorenz-SH-attractor} can be compared to the local large deviations principle established by Rey-Bellet and Young \cite{RBY}.
The constant $c_\infty$, defined by~\eqref{cinfty}, measures tails of constants appearing in the concept of weak Gibbs measure. 
If $\mu_\psi$ is a (strong) Gibbs measure then: (i) $c_\infty=-\infty$ and the upper bound in the theorem reduces to 
$-\inf_{\mu \in \cK} \mathfrak I_{\psi}(\mu)$, and (ii) there is no need to restrict to probability measures supported on $\Lambda_H$ and the lower bound reduces to
$-\inf_{\mu \in \cO} \mathfrak I_{\psi}(\mu)$ (item~\ref{item:LD-lower-Gibbs}). 
The function $\mathfrak I_\psi(\mu)$ satisfying Theorem~\ref{Thm:LD2B-Lorenz-SH-attractor} is called the large deviations rate function.
Moreover, while large deviations for flows usually involve the use of local cross-sections and Poincar\'e maps, creating dynamical systems with non-compact phase spaces and discontinuities (see e.g.\cite{Ar07,AST19}),
our approach deals with the flows (and time-t maps) directly, hence it avoids creating such technical obstructions.

\begin{remark}\label{Rem:LD}
	We will prove two general results Theorem~\ref{Thm:irregular-levelset} and Theorem~\ref{Thm:LD2B} stating that when $\Lambda$ is a singular hyperbolic homoclinic class such that each pair of periodic orbits are homoclinically related and $\overline{\mathcal M_{1}(\Lambda)}=\mathcal M_{inv}(\Lambda)$, then the conclusions of Theorem~\ref{Thm:irregular-Lorenz},~\ref{Thm:irregular-SinHyp-attractor} and~\ref{Thm:LD2B-Lorenz-SH-attractor} holds. Then Theorem~\ref{Thm:irregular-Lorenz} and~\ref{Thm:irregular-SinHyp-attractor} are direct consequences of Theorem~\ref{Thm:irregular-levelset} and~\ref{Thm:LD2B-Lorenz-SH-attractor} is a direct consequence of Theorem~\ref{Thm:LD2B}. The reason is that when 
	$\Lambda$ is a Lorenz attractor of vector fields in a Baire residual subset $\mathcal{R}^r\subset\mathcal{X}^r(M^3), (r\in\mathbb{N}_{\geq 2})$ or $\Lambda$ is a singular hyperbolic attractor $\Lambda$ of vector fields in a Baire residual set $\mathcal{R}\subset\mathcal{X}^1(M)$,
	then $\Lambda$ is a homoclinic class such that each pair of periodic orbits are homoclinically related (cf.~\cite[Theorem 6.8]{AP} for Lorenz attractors and~\cite[Theorem B]{CY-Robust-attractor} for singular hyperbolic attractors) and $\overline{\mathcal M_{1}(\Lambda)}=\mathcal M_{inv}(\Lambda)$ (cf.~\cite[Theorem A \& B]{STW}). See also Proposition~\ref{Pro:lorenz-robust} and Corollary~\ref{Pro:HS-approximation-generic} in this paper.
\end{remark}


\section{Entropy denseness and the horseshoe approximation property}\label{Section:ED-HS-approximation}

We give a criterion to study  multifractal analysis and large deviations for flows beyond uniform hyperbolicity.

\subsection{Entropy denseness of horseshoes}
Recall that for each invariant compact set $\Lambda$ of a vector field $X\in\mathscr{X}^1(M)$, we denote by $\mathcal{M}_{inv}(\Lambda)$ and $\mathcal{M}_{erg}(\Lambda)$ the space of invariant and ergodic probability measures supported on $\Lambda$, respectively, and that $d^*$ is a translation invariant metric
on $\mathcal M(\Lambda)$ compatible with the weak$^*$ topology.

\begin{definition}\label{Def:entropy-dense}
	Given $X\in\mathscr{X}^1(M)$ and an invariant compact  set $\Lambda$. 
	We say a convex subset $\mathcal{M}\subseteq \mathcal{M}_{inv}(\Lambda)$ is \emph{entropy-dense} if for any $\varepsilon>0$ and any $\mu\in \mathcal{M}$, there exists $\nu\in \mathcal{M}_{erg}(\Lambda)$ satisfying 
	$$d^*(\mu,\nu)<\varepsilon~~~\text{and}~~~ h_{\nu}(X)>h_{\mu}(X)-\varepsilon.$$
\end{definition}

It follows from the definition that the entropy denseness property is hereditary, i.e. if $\mathcal M_1\subset \mathcal M_2 \subset \mathcal M_{inv}(\Lambda)$ are
convex and $\mathcal M_2$ is entropy dense, so is $\mathcal M_1$. 
In what follows we discuss some consequences of approximating invariant sets by horseshoes. 
In particular the strongest conclusion is entropy denseness of the entire space $\mathcal M_{inv}(\Lambda)$.
We now recall the definition of horseshoe.

\begin{definition}\label{Def:horseshoe}
	Given $X\in\mathscr{X}^1(M)$, 
	an invariant compact set $\Lambda$ is called a \emph{basic set} if  it is, transitive, hyperbolic and  locally maximal, i. e. there exists an open neighborhood $U$ of $\Lambda$, such that $\Lambda=\cap_{t\in\RR}\phi_t(U)$.
	A basic set $\Lambda$ is called a \emph{horseshoe} if $\Lambda$ is a proper subset of $M$, is not reduced to a single orbit of a hyperbolic critical element
	and its intersection with any local cross-section to the flow is totally disconnected.
\end{definition}

Following the classical arguments of Bowen~\cite{Bowen72,Bowen73} on 
Axiom A vector fields,  every horseshoe is semi-conjugate to the suspension of a transitive subshift of finite type (SFT) with a continuous roof function through a finite-to-one continuous map.  See also a detailed explanation in~\cite[Section 2.4]{BS00b}. We formulate this as the following theorem.

\begin{theorem}[Bowen]\label{Thm:HS-semiconjugate-suspension}
	Assume $\Lambda$ is a horseshoe 
	of $X\in\mathcal{X}^1(M)$. Then there exists a transitive subshift of finite type $(\Sigma,\sigma)$ and a continuous roof function $\rho\colon \Sigma\rightarrow\mathbb{R}^+$, such that $(\Lambda,\Phi)$ is semi-conjugate to the suspension $(\Sigma_\rho,\F)$ through a continuous surjective map $\pi\colon \Sigma_\rho\rightarrow \Lambda$, where $\Phi=(\phi_t)_t$ is the flow generated by $X$ and $\F=(\sigma_t)_t$ is the suspension flow. Moreover, the semi-conjugacy $\pi$ is finite-to-one and hence preserves entropy: 
	\begin{enumerate}
		\item for any $\mu\in\mathcal{M}_{inv}(\Sigma_\rho,\F)$, there exists a unique $\nu\in\mathcal{M}_{inv}(\Lambda)$ satisfying $\mu=\pi^*(\nu)$ and their metric entropies coincide $h_{\mu}(\F)=h_{\nu}(X)$;
		\item for any invariant set $Z$ of $(\Sigma_\rho,\F)$, it satisfies $h_{\rm top}(Z,\F)=h_{\rm top}(\pi(Z),X)$;
	\end{enumerate}
\end{theorem}

In the 1970's, Sigmund~\cite{Sigmund70,Sigmund} studied the space of invariant measures of basic sets for Axiom A systems. The following theorem is from~\cite[Theorems 2\&3]{Sigmund}. 

\begin{theorem}[Sigmund]\label{Thm:HS-Sigmund}
	Assume $\Lambda$ is a horseshoe of $X\in\mathcal{X}^1(M)$. The following properties are satisfied.
	\begin{enumerate}
		\item\label{item:mu-generic-HS} For any $\mu\in\mathcal{M}_{inv}(\Lambda)$, the set 
$\displaystyle G_{\mu}:=\left\{x\in M \colon \lim_{t\rightarrow \infty}\frac{1}{t}\int_{0}^{t}\delta_{\phi_s(x)} {\rm d}s=\mu \right\}$
of $\mu$-generic points is non-empty;		
		\item\label{item:generic-measure-HS} There exists a residual subset $\mathcal{G}\subset\mathcal{M}_{inv}(\Lambda)$ such that each $\mu\in\mathcal{G}$ is ergodic and $\supp(\mu)=\Lambda$.
	\end{enumerate}
\end{theorem}

\begin{remark*}
	Although Sigmund's original statements were for basic sets, his arguments could be easily applied to horseshoes (isolated hyperbolic non-trivial transitive sets). See also remarks after~\cite[Theorem 3.5]{abc}
\end{remark*}

The following lemma states that two horseshoes are contained in a larger one once they are homoclinically related. The proof applies the $\lambda$-lemma~\cite[Lemma 7.1]{Palis-de Melo} (also see~\cite[Theorem 5.1]{Wen-book}) and is classical, thus we omit it.
\begin{lemma}\label{Lem:large-horseshoe}
	Let $\Lambda_1$ and $\Lambda_2$ be two horseshoes of $X\in\mathcal{X}^1(M)$. Assume there exists hyperbolic periodic points $p_1\in\Lambda_1$ and $p_2\in\Lambda_2$ such that $\orb(p_1)$ and $\orb(p_2)$ are homoclinically related. Then there exists a larger horseshoe $\Lambda$ that contains both $\Lambda_1$ and $\Lambda_2$.
\end{lemma}

The next proposition ensures that horseshoes are entropy-dense.

\begin{proposition}\label{Pro:horseshoe-entropy-dense}
Assume that $X\in\mathscr{X}^1(M)$. If $\Lambda$ is a horseshoe then $\mathcal{M}_{inv}(\Lambda)$ is entropy-dense. 
\end{proposition}

\begin{proof}
Although the result is probably known we shall include a proof as we could not find a reference. 
Assume that $\Lambda$ is a horseshoe, by Theorem~\ref{Thm:HS-semiconjugate-suspension}, it is semiconjugate to a suspension flow $\F$ over a transitive subshift of finite type $(\Sigma,\sigma)$ with a continuous roof function $\rho$. As the semi-conjugacy preserves entropy, we may deal directly with the case of the suspension flow $\F$. Moreover, since the entropy map 
is upper-semicontinuous and $\sigma$ satisfies the gluing orbit property (also known as transitive specification), 
Theorem~B in \cite{EKW94} guarantees that $\sigma$ is entropy-dense. 

Fix $\varepsilon>0$ and an arbitrary $\F$-invariant probability measure $\mu_\rho$, determined by a $\sigma$-invariant probability measure $\mu$ (recall Subsection~\ref{subsec:suspension}). 
Take $\delta>0$ small, to be determined later. Since the subshift of finite type $\sigma$ is entropy-dense, one picks a $\sigma$-invariant and ergodic probability $\nu$ so that 
\begin{itemize}
	\item $d^*(\mu,\nu)$ is small enough such that $$1-\delta<\frac{\int\rho {\rm d}\mu}{\int\rho {\rm d}\nu}<1+\delta.$$
	\item $h_{\nu}(\sigma)>h_{\mu}(\sigma)-\delta.$
\end{itemize}

Then the Abramov formula for the $\F$-invariant probabilities $\mu_\rho$ and $\nu_\rho$ (recall Subsection~\ref{subsec:suspension}) ensures that 
$$\displaystyle
h_{\nu_{\rho}}(\F)=\frac{h_{\nu}(\sigma)}{\int\rho {\rm d}\nu} 
	> \frac{\int\rho {\rm d}\mu}{\int\rho {\rm d}\nu} \cdot \Big(h_{\mu_{\rho}}(\F) -\frac{\delta}{\int\rho {\rm d}\mu}\Big) 
	>(1-\delta)\cdot \Big(h_{\mu_{\rho}}(\F) -\frac{\delta}{\int\rho {\rm d}\mu}\Big) 
	> h_{\mu_{\rho}}(\F) - \varepsilon 
$$
provided that $\delta$ is small enough. Diminishing $\delta$ if necessary we may also get that $\mu_\rho$ and $\nu_\rho$ are also $\varepsilon$-close in the $d^*$-metric.  This completes the proof of the proposition.
\end{proof}

\subsection{Horseshoe approximation property for singular hyperbolic homoclinic classes}\label{Sec:HS-approximation}
We introduce a notion of \emph{horseshoe approximation property}, a condition stronger than
the entropy-denseness condition in Definition~\ref{Def:entropy-dense}. 


\begin{definition}\label{Def:HS-appro-II}
	Given $X\in\mathscr{X}^1(M)$ and an invariant compact set $\Lambda$,
	we say a convex subset $\mathcal{M}\subseteq \mathcal{M}_{inv}(\Lambda)$ has the \emph{horseshoe approximation property} if for each $\varepsilon>0$ and any $\mu\in \mathcal{M}$, 
	there exist a horseshoe $\Lambda'\subset\Lambda$ and $\nu\in\mathcal{M}_{erg}(\Lambda')$ 
	\begin{equation}\label{eq:defHSSapproII}
	d^*(\nu,\mu)<\varepsilon ~~~\text{and}~~~h_{\nu}(X)>h_{\mu}(X)-\varepsilon.
	\end{equation}
\end{definition}

\begin{remark}\label{Rem:HS-approximation}
	By Proposition~\ref{Pro:horseshoe-entropy-dense}, if $\Lambda$ is a horseshoe, then $\mathcal{M}_{inv}(\Lambda)$ admits the horseshoe approximation property naturally.
	Moreover, it is clear from \eqref{eq:defHSSapproII} in the previous definition, 
that the horseshoe $\Lambda'$ 
satisfies $h_{\rm top}(\Lambda')>h_{\mu}(X)-\varepsilon$. 
%
Finally, in some specific contexts the approximating horseshoe $\Lambda'$ can be constructed using an analogue of 
Katok's arguments~\cite{Katok} and all measures supported on $\Lambda'$ are within an $\varepsilon$-neighborhood (in the 
weak$^*$ topology) of the original probability $\mu$ (see e.g. Lemma~\ref{Lem:Horseshoe-entropy-ergodic}).

\end{remark}


The horseshoe approximation property will be essential in the technique to deal with multifractal analysis. 
We proceed to analyse the horseshoe approximation in the context of singular hyperbolic sets.
Given a compact and invariant set $\Lambda$ of $X\in\mathscr{X}^1(M)$, denote by $\mathcal{M}_{per}(\Lambda)$ the set of periodic measures supported on $\Lambda$ and set   $$\mathcal{M}_1(\Lambda)=\left\{\mu\in \mathcal{M}_{inv}(\Lambda): \mu(\sing(\Lambda))=0 \right\} ~~\text{and}~~ \mathcal{M}_0(\Lambda)=\mathcal{M}_{erg}(\Lambda)\cap \mathcal{M}_1(\Lambda).$$ 
We first prove the following auxiliary lemma.

\begin{lemma}\label{Lem:periodic measure dense}
	Let $X\in\mathscr{X}^1(M)$ and $\Lambda$ be a singular hyperbolic homoclinic class of $X$. If each pair of periodic orbits of $\Lambda$ are homoclinic related, then $$\overline{\mathcal{M}_{per}(\Lambda)}=\overline{\emph{Convex}(\mathcal{M}_0(\Lambda))}=\overline{\mathcal{M}_1(\Lambda)},$$
	where $\emph{Convex}(\mathcal{M}_0(\Lambda))$ is the convex hull of $\mathcal{M}_0(\Lambda)$.
\end{lemma}

\begin{proof}
	As the second equality above 
	is immediate from the ergodic decomposition theorem we are left to prove the first one. By~\cite[Proposition 3.1]{STW},   $\mathcal{M}_0(\Lambda)\subset \overline{\mathcal{M}_{per}(\Lambda)}$. Since $\Lambda$ is a homoclinic class, the set $\overline{\mathcal{M}_{per}(\Lambda)}$ is convex by~\cite[Proposition 4.7\& Remark 4.6]{abc}. Hence we have that 
	$\overline{\text{Convex}(\mathcal{M}_0(\Lambda))}\subset \overline{\mathcal{M}_{per}(\Lambda)}$. The inclusion 
	 $\overline{\mathcal{M}_{per}(\Lambda)}\subset\overline{\text{Convex}(\mathcal{M}_0)}$ is obvious since 
	 $\mathcal{M}_{per}(\Lambda)\subset\mathcal{M}_0(\Lambda)$.
\end{proof}

	Given $X\in\mathscr{X}^1(M)$, let $\Lambda$ be an invariant compact  set displaying a singular hyperbolic splitting $T_{\Lambda}M=E^{ss}\oplus E^{cu}$. Then for any ergodic measure $\mu\in\mathcal{M}_{0}(\Lambda)$, the splitting $E^{ss}\oplus E^{cu}|_{\supp(\mu)}$ is a dominated splitting and the index of $\mu$ equals $\dim(E^{ss})$ obviously. Using Katok's arguments in~\cite[Theorem 4.3]{Katok} one obtains the following lemma.
    Given $X\in\mathscr{X}^1(M)$ and assume $\Lambda$ is an invariant compact set. We say that $X$ satisfies the {\it star} condition on $\Lambda$ if there exist a neighborhood $U$ of $\Lambda$ and a $C^1$-neighborhood $\mathcal{U}$ of $X$ in $\mathscr{X}^1(M)$ such that every periodic orbit contained in $U$ associated to a vector field $Y\in\mathcal{U}$  is hyperbolic.

\begin{lemma}\label{Lem:Horseshoe-entropy-ergodic}
	Let $X\in\mathscr{X}^1(M)$ and $\Lambda$ be a singular hyperbolic homoclinic class of $X$. 
	Assume each pair of periodic orbits of $\Lambda$ are homoclinic related, and $\mu\in\mathcal{M}_{0}(\Lambda)$. 
	Then for any $\varepsilon>0$ there exists a horseshoe $\Lambda_{\varepsilon}\subset\Lambda$ 
	contained in the $\varepsilon$-neighborhood of $\supp(\mu)$ (in the Hausdorff distance), and so that
	$d^*(\mu,\nu)<\varepsilon$ for any $\nu\in\mathcal{M}_{inv}(\Lambda_\varepsilon)$, and there exists $\nu_0\in\mathcal{M}_{erg}(\Lambda_\varepsilon)$ satisfying 
	$h_{\nu_0}(X)>h_{\mu}(X)-\varepsilon$.
	
\end{lemma}

\begin{proof}
	We only give a sketch here since the proof is essentially contained in~\cite[Proposition 2.9]{lsww}
    (see also~\cite[Theorem 4.1]{PYY}).
	Note that since $\Lambda$ is a singular hyperbolic homoclinic class, the vector field $X$ satisfies the star condition  over $\Lambda$. 
    The existence of a horseshoe $\Lambda_{\varepsilon}\subset\Lambda$ satisfying 
    $h_{\rm top}(\Lambda_{\varepsilon})>h_{\mu}(X)-\varepsilon$ follows directly from~\cite[Proposition 2.9]{lsww}. 
	Moreover, the horseshoe  $\Lambda_{\varepsilon}$ is constructed by shadowing the orbit of a generic point of $\mu$, following the classical arguments of Katok~\cite[Theorem 4.3]{Katok}, hence $\Lambda_\varepsilon$ is contained in the $\varepsilon$-neighborhood of $\supp(\mu)$ and	
	every invariant measure $\nu$  supported on $\Lambda_{\varepsilon}$ is close to $\mu$ in the weak$^*$-topology. Finally, by the variational 
	principle and the fact  that
	$h_{\rm top}(\Lambda_{\varepsilon})>h_{\mu}(X)-\varepsilon$, there exists $\nu_0\in\mathcal{M}_{erg}(\Lambda_\varepsilon)$ satisfying	$h_{\nu_0}(X)>h_{\mu}(X)-\varepsilon$.
\end{proof}

\begin{remark}\label{Rem:conjugancy-of-HS}
	In fact, the horseshoe $\Lambda_{\varepsilon}$ obtained in Lemma~\ref{Lem:Horseshoe-entropy-ergodic} is conjugate (not only semi-conjugate) to the suspension flow of a full shift 
	with continuous roof function. We refer the reader to~\cite[Proposition 2.9]{lsww} for more details and to~\cite[Theorem 5.6]{sgw} for an approach.
\end{remark}


The approximation by horseshoes in the conclusion of Lemma~\ref{Lem:Horseshoe-entropy-ergodic} can actually be extended to arbitrary invariant measures in $\mathcal{M}_{1}(\Lambda)$. More precisely: 

\begin{proposition}\label{Pro:Horseshoe-entropy-invariant}
	Let $X\in\mathscr{X}^1(M)$ and $\Lambda$ be a singular hyperbolic homoclinic class of $X$. 
	If each pair of periodic orbits of $\Lambda$ are homoclinic related, then for every $\mu\in\mathcal{M}_1(\Lambda)$ and $\varepsilon>0$, there exists a horseshoe $\Lambda_{\varepsilon}\subset\Lambda$ and 
	there exists $\nu\in\mathcal{M}_{erg}(\Lambda_{\varepsilon})$ satisfying $$d^*(\mu,\nu)<\varepsilon~~\text{and}~~h_{\nu}(X)>h_{\mu}(X)-\varepsilon.$$
Thus ${\mathcal{M}_1(\Lambda)}$ has the horseshoe approximation property, and so it is entropy-dense.
\end{proposition}
\begin{proof}
	Let $\mu\in\mathcal{M}_{1}(\Lambda)$ and $\varepsilon>0$ be fixed. By ergodic decomposition and affinity of the metric entropy, there exist ergodic measures $\omega_1,\omega_2,\cdots,\omega_k\in \mathcal{M}_{0}(\Lambda)$ and real numbers $\alpha_1,\alpha_2,\cdots,\alpha_k\in (0,1)$ with $\sum\limits_{i=1}^k\alpha_i=1$ 	so that the probability $\omega'=\sum\limits_{i=1}^k\alpha_i\omega_i$ satisfies  
	$$
	d^*(\omega',\mu)<\frac{\varepsilon}{3}~~\text{and}~~h_{\omega'}(X)>h_{\mu}(X)-\frac{\varepsilon}{3}.
	$$
	By Lemma~\ref{Lem:Horseshoe-entropy-ergodic}, for each $i\in\{1,2,\cdots,k\}$, there exists a horseshoe $\Lambda_i\subset\Lambda$ such that 
	every $\omega\in\mathcal{M}_{erg}(\Lambda_i)$ satisfies $d^*(\omega_i,\omega)<\frac{\varepsilon}{6}$,
	and there exists $\nu_i\in\mathcal{M}_{erg}(\Lambda_i)$ satisfying 
	$$
	d^*(\omega_i,\nu_i)<\frac{\varepsilon}{3}~~\text{and}~~h_{\nu_i}(X)>h_{\omega_i}(X)-\frac{\varepsilon}{3}.
	$$
	Let $\Lambda_{\varepsilon}$ be a large horseshoe that contains every $\Lambda_i$ for $i\in\{1,2,\cdots,k\}$. Such $\Lambda_{\varepsilon}$ does exist since any two periodic orbits in $\Lambda$ are homoclinically related and any horseshoe must contain (countably many) periodic orbits. Then the probability measure $\nu'=\sum\limits_{i=1}^k\alpha_i\nu_i$ in $\mathcal{M}_{inv}(\Lambda_{\varepsilon})$ satisfies
	$d^*(\nu',\mu)\leq d^*(\nu',\omega')+d^*(\omega',\mu)<\frac{2\varepsilon}{3},$
	and
	$$h_{\nu'}(X)=\sum\limits_{i=1}^k\alpha_i\cdot h_{\nu_i}(X)>\sum\limits_{i=1}^k\alpha_i\cdot h_{\omega_i}(X)-\frac{\varepsilon}{3}>h_{\mu}(X)-\frac{2\varepsilon}{3}.$$
	By Proposition~\ref{Pro:horseshoe-entropy-dense}, we know that $\mathcal{M}_{inv}(\Lambda_{\varepsilon})$ is entropy-dense. Thus, for the invariant measure $\nu'\in \mathcal{M}_{inv}(\Lambda_{\varepsilon})$ and $\varepsilon>0$ above, there exists  $\nu\in\mathcal{M}_{erg}(\Lambda_{\varepsilon})$ satisfying
	$$d^*(\nu',\nu)<\frac{\varepsilon}{3}~~\text{and}~~h_{\nu}(X)>h_{\nu'}(X)-\frac{\varepsilon}{3}.$$
	In consequence,  the ergodic probability $\nu$ supported on $\Lambda_\varepsilon$ 
	satisfies
	$$d^*(\mu,\nu)\leq d^*(\mu,\nu')+d^*(\nu',\nu)<\varepsilon,$$
	and
	$$h_{\nu}(X)>h_{\nu'}(X)-\frac{\varepsilon}{3}>h_{\mu}(X)-\varepsilon.$$
\end{proof}

\begin{proposition}\label{Pro:Horseshoe-approximation-inv}
	Let $X\in\mathscr{X}^1(M)$ and $\Lambda$ be a singular hyperbolic homoclinic class of $X$. 
	Assume each pair of periodic orbits of $\Lambda$ are homoclinic related, and $\mathcal{M}_{inv}(\Lambda)=\overline{\mathcal{M}_1(\Lambda)}$.  
	Then $\mathcal{M}_{inv}(\Lambda)$ has the horseshoe approximation property.
\end{proposition}

\begin{proof}
Fix an arbitrary $\mu \in \mathcal{M}_{inv}(\Lambda)=\overline{\mathcal{M}_1(\Lambda)}$. 
Using Proposition~\ref{Pro:Horseshoe-entropy-invariant} it suffices to show that 
$\mu$ is approximated, both in the weak$^*$ topology and entropy, by measures in ${\mathcal{M}_1(\Lambda)}$.

Fix an arbitrary constant $\varepsilon>0$.
If $\mu \in {\mathcal{M}_1(\Lambda)}$
there is nothing to prove. Otherwise,  one can write $\mu=\alpha \mu_1 + (1-\alpha) \mu_2$
for some $0< \alpha \leq 1$ and probabilities $\mu_1,\mu_2 \in \mathcal{M}_{inv}(\Lambda)$ so that $\mu_1(\sing(X))=1$
and $\mu_2(\sing(X))=0$. In other words, $\mu_1$ is supported in the invariant set formed by  singularities and $\mu_2\in \mathcal M_1(\Lambda)$.
By the assumption, there exists a sequence of probabilities $\nu_n \in {\mathcal{M}_1(\Lambda)}$ with 
$d^*(\nu_n,\mu_1)\rightarrow 0$ as $n\rightarrow \infty$. 
Notice that $\mathcal{M}_1(\Lambda)$ is a convex set.
Thus $\alpha \nu_n + (1-\alpha) \mu_2 \in \mathcal{M}_1(\Lambda)$ for each $n\ge 1$.
Moreover, using that $d^*$ is translation invariance and affinity of the entropy we get
\begin{align*}
d^*\big(\alpha \nu_n + (1-\alpha) \mu_2,\mu\big) 
	 = d^*\big(\alpha \nu_n + (1-\alpha) \mu_2,\alpha \mu_1 + (1-\alpha) \mu_2\big) 
	 = d^*\big(\alpha \nu_n,\alpha \mu_1\big) 
\end{align*}
and
$$
h_{\alpha \nu_n + (1-\alpha) \mu_2}(X)
	= \alpha h_{\nu_n}(X) + (1-\alpha) h_{ \mu_2}(X)\ge (1-\alpha) h_{ \mu_2}(X) = h_\mu(X).
$$ 
By taking $n\ge 1$ large, we conclude that the probability $\mu_n\colon=\alpha\nu_n+(1-\alpha)\mu_2\in \mathcal M_1(\Lambda)$ satisfies 
$
	d^*(\mu_n,\mu)<{\varepsilon}~~\text{and}~~h_{\mu_n}(X) \ge h_{\mu}(X).
$
This completes the proof. 
\end{proof}

Recently, S. Crovisier and D. Yang~\cite{CY-Robust-attractor} proved that for $C^1$-open and dense set of vector field $X\in\mathcal{X}^1(M)$, any singular hyperbolic attractor $\Lambda$ is a robustly transitive attractor. Moreover, if $\Lambda$ is non-trivial, then it is a homoclinic class and any two periodic orbits contained in $\Lambda$ are homoclinically related.
On the other hand, the main theorems (Theorem A, B, B') in~\cite{STW}  state that if $\Lambda$ is a singular hyperbolic attractor of $X$ in a residual subset $\mathcal{R}\subset\mathscr{X}^1(M)$, or  $\Lambda$ is a geometric Lorenz attractor of $X$ in a residual subset $\mathcal{R}^r\subset\mathscr{X}^r(M^3)$, then $\mathcal{M}_{inv}(\Lambda)=\overline{\mathcal{M}_{per}(\Lambda)}$, which implies naturally that $\mathcal{M}_{inv}(\Lambda)=\overline{\mathcal{M}_1(\Lambda)}$.
Thus one obtains the following consequence from Proposition~\ref{Pro:Horseshoe-approximation-inv}.

\begin{corollary}\label{Pro:HS-approximation-generic}
The following holds:
	\begin{enumerate}
		\item\label{item:generic-Lorenz-attractor} There exists a residual subset  $\mathcal{R}^r\subset\mathscr{X}^r(M^3)$ where $r\in\mathbb{N}_{\geq 2}\cup\{\infty\}$ such that if $\Lambda$ is a geometric Lorenz attractor of $X\in\mathcal{R}^r $, then $\mathcal{M}_{inv}(\Lambda)=\overline{\mathcal{M}_1(\Lambda)}$ and thus $\mathcal{M}_{inv}(\Lambda)$ has the horseshoe approximation property and is entropy-dense.
		\item\label{item:generic-SH-attractor} There exists a residual subset  $\mathcal{R}\subset\mathscr{X}^1(M)$ such that if $\Lambda$ is a non-trivial singular hyperbolic attractor of $X\in\mathcal{R}$, then $\mathcal{M}_{inv}(\Lambda)=\overline{\mathcal{M}_1(\Lambda)}$ and thus $\mathcal{M}_{inv}(\Lambda)$ has the horseshoe approximation property and is entropy-dense.  
	\end{enumerate}
\end{corollary}

\vskip2mm

The following proposition, whose strong conclusion will not be used in full strength in this paper 
guarantees that the entropy of a singular hyperbolic homoclinic class can be approximated by a horseshoe supporting ergodic measures which are dense enough. More precisely:

\begin{proposition}\label{Pro:Horseshoe-approximation}
	Let $X\in\mathscr{X}^1(M)$ and $\Lambda$ be a singular hyperbolic homoclinic class of $X$. 
	If each pair of periodic orbits of $\Lambda$ are homoclinic related, then for every $\mu\in \overline{\mathcal{M}_1(\Lambda)}$ and $\varepsilon>0$, there exist a horseshoe $\Lambda_\varepsilon\subseteq\Lambda$ and  $\nu\in \mathcal{M}_{erg}(\Lambda_\varepsilon)$	so that $h_{\rm top}(\Lambda_\varepsilon)> h_{\rm top}(\Lambda)-\varepsilon$ and $d^*(\nu,\mu)<\varepsilon$.
\end{proposition}
\begin{proof}
	In the special case that $\Lambda\cap\sing(X)=\emptyset$ we have that $\Lambda$ hyperbolic (recall Remark~\ref{Rem:singular hyperbolicity}), hence $\Lambda$  itself is a basic set of $X$. Moreover, if this is the case then 
	$\mathcal{M}_{inv}(\Lambda)=\mathcal{M}_{1}(\Lambda)=\overline{\mathcal{M}_{per}(\Lambda)}$ by~\cite[Theorem 1]{Sigmund}.
	Thus $\overline{\mathcal{M}_1(\Lambda)}$ has the strong horseshoe approximation property by Proposition~\ref{Pro:Horseshoe-approximation-inv} and one concludes.

	\smallskip
	It remains to consider the case where $\Lambda\cap\sing(X)\neq\emptyset$.
    We first construct a nested sequence of horseshoes whose entropies approximate to $h_{\rm top}(\Lambda)$. Notice first that every periodic orbit in $\Lambda$ is hyperbolic, hence the non-trivial homoclinic class $\Lambda$ contains countably many periodic orbits which we list as 
	$\{\gamma_n\}_{n\geq 1}$. Moreover, $\Lambda$ being a non-trivial homoclinic class ensures that $h_{\rm top}(\Lambda)>0$. By the variational principle and Lemma~\ref{Lem:Horseshoe-entropy-ergodic}, there is a sequence of horseshoes $\{\Delta_n\}_{n\geq 1}$ contained in $\Lambda$ such that $h_{\rm top}(\Delta_n)>h_{\rm top}(\Lambda)-1/n$. Notice that any two periodic orbits contained in $\Lambda$ have the same stable index and are homoclinically related with each other, 
	and each horseshoe $\Delta_n$ must contain infinitely many periodic orbits for each $n\geq 1$. Thus, inductively, we can construct a nested sequence of transitive horseshoes $\{\Lambda_n\}_{n\geq 1}$ contained in $\Lambda$  as follows:
	\begin{itemize}
		\item Let $\Lambda_1$ be a horseshoe that contains $\gamma_1$ and $\Delta_1$. Such $\Lambda_1$ does exist since $\gamma_1$ is homoclinically related with all periodic orbits contained in $\Delta_1$.
		\item For $n\geq 2$, let $\Lambda_n$ be a  horseshoe that contains $\gamma_n$ and also contains the two horseshoes $\Lambda_{n-1}$ and $\Delta_n$. Such a horseshoe exists by Lemma~\ref{Lem:large-horseshoe}.
	\end{itemize}
    Then we have that $h_{\rm top}(\Lambda_n)\geq h_{\rm top}(\Delta_n)>h_{\rm top}(\Lambda)-1/n$ for every $n\geq 1$.
  
   We claim that that the previous sequence of horseshoes $\{\Lambda_n\}_{n\in\NN}$ satisfies the conclusion.
   Indeed, for any $\varepsilon>0$ and $\mu\in \overline{\mathcal{M}_1(\Lambda)}$, by Lemma~\ref{Lem:periodic measure dense} there exists $n_1\in\mathbb{N}$ such that $d^*(\mu,\nu_{n_1})<\varepsilon$ where $\nu_{n_1}$ is the periodic measure associated to $\gamma_{n_1}$. Take $n_2\in\mathbb{N}$ such that $1/n_2<\varepsilon$ and $n=\max\{n_1,n_2\}$. Therefore $\nu_{n_1}\in \mathcal{M}_{erg}(\Lambda_n)$ and 
   $$h_{\rm top}(\Lambda_n)>h_{\rm top}(\Lambda)-1/n\geq h_{\mu}(X)-1/n>h_{\mu}(X)-\varepsilon.$$
This completes the proof of Proposition~\ref{Pro:Horseshoe-approximation}.
\end{proof}

\section{Multifractal analysis}\label{Section:proof-MA}
In this section, we aim to study the multifractal analysis of singular hyperbolic attractors of $C^1$-generic vector fields and geometric Lorenz attractors of $C^r$-generic vector fields ($r\geq 2$). We prove the following theorem in general case. With the arguments in Section~\ref{Sec:HS-approximation}, one will see that Theorems~\ref{Thm:irregular-Lorenz} \&~\ref{Thm:irregular-SinHyp-attractor} are consequences of Theorem~\ref{Thm:irregular-levelset} below
(cf. Subsection~\ref{sec:proofsAB}). The stra\-tegy is to use the horseshoe approximation property to transfer the difficulty to the des\-cription of suitably chosen horseshoes. 
Once this is accomplished, then one can use Thompson's results to get full topological entropy of  irregular sets~\cite{Thomp} and variational principle of level sets~\cite{Thompson2009}.


\begin{theorem}\label{Thm:irregular-levelset} 
	Let $X\in\mathscr{X}^1(M)$ and $\Lambda$ be a singular hyperbolic homoclinic class of $X$. Assume each pair of periodic orbits contained in $\Lambda$ are homoclinically related and $\mathcal{M}_{inv}(\Lambda)=\overline{\mathcal{M}_1(\Lambda)}$. Given $g\in C(\Lambda,\mathbb{R})$, then either	
	\begin{enumerate}
		\item\label{item:irr-empty} $I_g$ is empty and  $\displaystyle \int g {\rm d}\mu=\displaystyle \int g {\rm d}\nu$  for all $\mu,\nu\in\mathcal{M}_{inv}(\Lambda)$; or
		\item\label{item:irr-residual} $I_{g}$ is residual in $\Lambda$ and $h_{\rm top}(I_{g})=h_{\rm top}(\Lambda)$. 
	\end{enumerate}
	Moreover, for each $a\in\RR$ satisfying 
		$
		\displaystyle\inf_{\mu\in \mathcal{M}_{1}(\Lambda)}\int g {\rm d}\mu<a<\sup_{\mu\in \mathcal{M}_{1}(\Lambda)} \int g {\rm d}\mu,
		$  
		the level set $R_{g}(a)$ is dense in $\Lambda$ and 
		$$h_{\rm top}(R_{g}(a))=\sup\left\{h_\mu(f) \colon \int g {\rm d}\mu=a, \mu\in   \mathcal{M}_{inv}(\Lambda)\right\}.$$ 
Furthermore the set of functions satisfying the second item form an open and dense subset in $C(\Lambda,\mathbb{R})$.
\end{theorem}

\begin{remark}\label{Rem:level-set}
	\begin{enumerate}
		\item 	Similarly to Remark~\ref{Rem:Thm-A-B}, for the characterizations of level sets in Theorem~\ref{Thm:irregular-levelset}, when 	
		$$a=\displaystyle\inf_{\mu\in \mathcal{M}_{inv}(\Lambda)}\int g {\rm d}\mu  \text{~~~or~~~}  a=\displaystyle\sup_{\mu\in \mathcal{M}_{inv}(\Lambda)} \int g {\rm d}\mu,$$	
		by~\cite{FH10}, one also has the variational principle that 	
		$$h_{\rm top}(R_{g}(a))=\sup\left\{h_\mu(X) \colon \int g {\rm d}\mu=a, \mu\in   \mathcal{M}_{inv}(\Lambda)\right\}.$$
		
		\item In~\cite[Section 9]{WangTY}, the author also constructed an increasing sequence of basic sets to study the multifractal spectra for Katok maps.
\end{enumerate}
\end{remark}

\subsection{Entropy estimates for irregular sets and level sets}
Given a compact metric space $(K,d)$, for a homeomorphism  $f\colon K\rightarrow K$ and a continuous roof function $\rho\colon K\rightarrow (0,+\infty)$, we consider the suspension flow $(K_{\rho},\F)$ where $\F=(f_t)_t$ as defined in Section~\ref{subsec:suspension}.
Analogous to the discrete case, for a continuous function $g\colon K_{\rho}\rightarrow \R$, we define
$$\underline{g}(x,s)=\liminf_{T\to\infty}\frac1T\int_0^Tg(f_t(x,s)){\rm d}t
\qquad \textrm{and} \qquad
\overline{g}(x,s)=\limsup_{T\to\infty}\frac1T\int_0^Tg(f_t(x,s)){\rm d}t.$$
Define the irregular set 
$$ I^{\rho}_{g}(\F):=\big\{(x,s)\in K_{\rho}\colon \underline{g}(x,s)< \overline{g}(x,s)\big\},$$ 
and for $a\in\R$, define the level set 
$$R^{\rho}_{g}(\F,a)=\big\{(x,s)\in K_{\rho}\colon \underline{g}(x,s)= \overline{g}(x,s)=a\big\}.$$

For a dynamical system $(K,f)$ satisfying the specification property, Thompson proved the following variational principle of level sets \cite[Theorem 4.2]{Thompson2009} and full topological entropy of irregular sets \cite[Theorem 5.1]{Thomp} for the suspension flow $(K_{\rho},\F)$ of $(K,f)$.

\begin{theorem}[Thompson~\cite{Thomp},~\cite{Thompson2009}]\label{Pro-Thompson}
Let $(K,d)$ be a compact metric space, $f\colon K\rightarrow K$ be a homeomorphism satisfying  the specification property and $\rho\colon K\to (0,+\infty)$ be a continuous function. Let $(K_{\rho},\F)$ denote the suspension flow over $(K,f)$ with roof function $\rho$ and let  $g\colon K_{\rho}\rightarrow \R$ be a continuous function. Then:
    \begin{enumerate}
       \item\label{item:levelset-suspension} For any $a\in\R$, 
       $$h_{\rm top}(\F,R^{\rho}_{g}(\F,a))=\sup\left\{h_{\mu}(\F)\colon\mu \in\mathcal{M}_{inv}(\F,K_{\rho})~\textrm{and}~\int g {\rm d}\mu=a\right\}.$$
       \item\label{item:irregular-suspension}  If  
       $\displaystyle\inf_{\mu\in \mathcal{M}_{inv}(\F,K_{\rho})}\int g {\rm d}\mu<\sup_{\mu\in\mathcal{M}_{inv}(\F,X_{\rho})}\int g {\rm d}\mu$
       then
       $$h_{\rm top}(\F,I^{\rho}_{g}(\F))=h_{\rm top}(\F).$$
    \end{enumerate}
\end{theorem}

\begin{remark}\label{Rem:suspension-SFT-Thompson}
	Note that a suspension flow over a transitive subshift of finite type (SFT) is topologically conjugate to a suspension flow over a topologically mixing SFT, and a topologically mixing SFT satisfies the specification property, thus the conclusions of Theorem~\ref{Pro-Thompson} hold for suspension flows over a transitive SFT.
\end{remark}

As a corollary, we have the following conclusion for horseshoes of $C^1$ vector fields, which complement previous results on the multifractal analysis of hyperbolic flows \cite{BS00b,CV21,PeSa}. Recall that for an invariant compact set $\Lambda$ of $X\in\mathcal{X}^1(M)$ and for a continuous function $g\colon\Lambda\rightarrow \mathbb{R}$, the $g$-irregular set is
$$
I_{g}:=\left\{x\in\Lambda \colon \lim_{T\rightarrow\infty}\frac{1}{T}\int_{0}^{T}g(\phi_t(x)) \,{\rm d}t \text{ does not exist}\right\}
$$ 
and, for each $a\in \mathbb{R}$, the $g$-level set is 
$$
R_{g}(a):=\left\{x\in\Lambda \colon \lim_{T\rightarrow\infty}\frac{1}{T}\int_{0}^{T}g(\phi_t(x)) {\rm d}t =a\right\}.
$$ 

\begin{corollary}\label{Cor:Thompson-HS}
	Let $\Lambda$ be a horseshoe of $X\in\mathcal{X}^1(M)$ and $g\colon\Lambda\rightarrow \mathbb{R}$ be a continuous function. The following properties hold:
		\begin{enumerate}
		\item\label{item:levelset-HS} For any $a\in\R$, 
		$$h_{\rm top}(R_{g}(a))=\sup\left\{h_{\mu}(X)\colon\mu \in \mathcal{M}_{inv}(\Lambda)~\textrm{and}~\int g {\rm d}\mu=a\right\}.$$
		\item\label{item:irregular-HS}  If   
		$\displaystyle\inf_{\mu\in \mathcal{M}_{inv}(\Lambda)}\int g {\rm d}\mu<\sup_{\mu\in\mathcal{M}_{inv}(\Lambda)}\int g {\rm d}\mu$, 
		then
		$h_{\rm top}(I_{g})=h_{\rm top}(\Lambda).$
	\end{enumerate}
\end{corollary}

\begin{proof}
	By Theorem~\ref{Thm:HS-semiconjugate-suspension}, there exists a suspension flow $(\Sigma_\rho,\F)$ over a transitive SFT $(\Sigma,\sigma)$ such that $(\Lambda,\Phi)$ is semi-conjugate to $(\Sigma_\rho,\F)$, where $\rho\colon \Sigma\rightarrow \mathbb{R}^+$ is a continuous function, $\Phi=(\phi_t)_t$ is the flow generated by $X$ and $\F=(\sigma_t)_t$ is the suspension flow. That is to say, there exists a finite-to-one continuous map $\pi\colon \Sigma_\rho\rightarrow \Lambda$ satisfying that $\pi\circ \sigma_t=\phi_t\circ\pi$ for any $t\in\mathbb{R}$. In particular $\pi$ preserves entropy.
	As $g\in C(\Lambda, \mathbb{R})$, then $\hat g=g\circ \pi\in C(\Sigma_\rho,\mathbb{R})$. Moreover, it is easy to check that 
	$\pi(R_{\hat g}^\rho(\F,a))=R_g(a)$ for each $a\in\mathbb{R}$, and $\pi(I_{\hat g}^\rho(\F))=I_g$. Thus Corollary~\ref{Cor:Thompson-HS} is a consequence 
	of Theorems~\ref{Thm:HS-semiconjugate-suspension} and Remark~\ref{Rem:suspension-SFT-Thompson}.
\end{proof}


Now we apply Theorem~\ref{Pro-Thompson} and Corollary~\ref{Cor:Thompson-HS} to singular hyperbolic homoclinic classes of $C^1$ vector fields.
 
\begin{proposition}\label{Pro:irregular} 
	Let $\Lambda$ be a singular hyperbolic homoclinic class of $X\in\mathcal{X}^1(M)$ such that each pair of periodic orbits are homoclinically related and $g\colon\Lambda\rightarrow \mathbb{R}$ be a continuous function.
	Assume that $\mathcal{M}\subseteq \mathcal{M}_{inv}(\Lambda)$ is a convex subset  satisfying the horseshoe approximation property and  
	$$\displaystyle\inf_{\mu\in\mathcal{M}}\int g {\rm d}\mu< \sup_{\mu\in\mathcal{M}}\int g {\rm d}\mu.$$  
	Then the following properties are satisfied.
	\begin{enumerate}
	\item\label{item:irregular-SH-Hclass} 
	The topological entropy of the $g$-irregular set $I_{g}$   satisfies
    $$h_{\rm top}(I_g)\geq h_{ \mathcal{M}}(X)\colon=\sup\left\{h_\mu(X)\colon  \mu\in   \mathcal{M}\right\}.$$ 
    \item\label{item:levelset-SH-Hclass} 
    For any $\displaystyle a\in\left(\inf_{\mu\in\mathcal{M}}\int g {\rm d}\mu, \sup_{\mu\in\mathcal{M}}\int g {\rm d}\mu\right)$, the  topological entropy of the level set $R_{g}(a)$  satisfies  
    $$h_{\rm top}(R_g(a))\geq h_{\mathcal{M}} (a)\colon=\sup\left\{h_\mu(X)\colon \mu\in   \mathcal{M} \text{~and ~}\int g {\rm d}\mu=a \right\}.$$ 
    \end{enumerate}
\end{proposition}

\begin{remark*}
It will be clear from the proof that for establishing item (2) above one uses exclusively the horseshoe approximation property assumption.
\end{remark*}

\begin{proof} 
	Denote by $\displaystyle\underline{a}=\inf_{\mu\in\mathcal{M}}\int g {\rm d}\mu \text{~and~} \overline{a}=\sup_{\mu\in\mathcal{M}}\int g {\rm d}\mu$ for simplicity.
	Fix $a\in(\underline{a},\overline{a})$.
	Since $\mathcal{M}$ is convex, one has that 
	$\left\{ \mu\in   \mathcal{M}  \colon \int g {\rm d}\mu=a \right\}\neq\emptyset.$
	For each $n\in\mathbb{N}$, take $\mu_{n},\nu_{n} \in \mathcal{M} $ so that
	\begin{itemize}
		\item $\displaystyle\int g {\rm d}\mu_{n}=a \text{~and~} h_{\mu_{n}}(X)> h_{ \mathcal{M}}(a) - \frac1n;$
  	    \item $h_{\nu_{n}}(X)> h_{ \mathcal{M}}(X) - \frac1n.$ 
	\end{itemize}

    We give the following claim first.
    
    \begin{claim}\label{Claim:horseshoe}
    	   There exists a horseshoe  $\Lambda_n\subseteq \Lambda$ such that 
    \begin{itemize}
    	\item[(a)] $\displaystyle\inf_{\mu\in \mathcal{M}_{inv}(\Lambda_n)}\int g {\rm d}\mu<\frac{\underline a+a}{2}<a<\frac{a+\overline a}{2}< \sup_{\mu\in\mathcal{M}_{inv}(\Lambda_n)}\int g {\rm d}\mu$;	
    	\item[(b)] there exist $\mu_n',\nu_n'\in \mathcal{M}_{inv}(\Lambda_n)$ such that \\
    	$\displaystyle(b.1)~ \left|\int g {\rm d}\mu_n-\int g {\rm d}\mu_n'\right|<\frac1n \text{~and~} h_{\mu_n'}(X)>h_{\mu_n}(X)-\frac1n>h_{ \mathcal{M}}(a) - \frac2n,$\\
    	$\displaystyle (b.2)~ h_{\nu_n'}(X)>h_{\nu_n}(X)-\frac1n>h_{ \mathcal{M}}(X) - \frac2n.$ 	
    \end{itemize}
    \end{claim}
    \begin{proof}[Proof of Claim~\ref{Claim:horseshoe}]
       The proof is by applying the fact that $\mathcal{M}$  satisfies the horseshoe approximation property twice and by Lemma~\ref{Lem:large-horseshoe}. We give a short explanation.
       First, by the definition of $\underline a$ and $\overline a$, there exist $\eta_1,\eta_2\in\mathcal{M}$ such that $\displaystyle\int g {\rm d}\eta_1<\frac{\underline a+a}{2}$ and $\displaystyle\int g {\rm d}\eta_2>\frac{\overline a+a}{2}$. Since $\mathcal{M}$  satisfies the horseshoe approximation property, there exist two horseshoes $\Delta_1,\Delta_2$ such that 
       $$\displaystyle\inf_{\mu\in \mathcal{M}_{inv}(\Delta_1)}\int g {\rm d}\mu<\frac{\underline a+a}{2}<a<\frac{a+\overline a}{2}< \sup_{\mu\in\mathcal{M}_{inv}(\Delta_2)}\int g {\rm d}\mu.$$
       On the other hand, 
       by the horseshoe approximation property again, there exist two horseshoes $\Delta_n^1,\Delta_n^2$ and $\mu_n'\in \mathcal{M}_{inv}(\Delta^1_n), \nu_n'\in \mathcal{M}_{inv}(\Delta^2_n)$ such that Items $(b.1)$ and $(b.2)$ are satisfied for $\mu_n'$ and $\nu_n'$. 
       Then by Lemma~\ref{Lem:large-horseshoe}, one can take a larger horseshoe $\Lambda_n$ that contains all the horseshoes $\Delta_1,\Delta_2,\Delta^1_n,\Delta^2_n$, thus the above statements hold for $\Lambda_n$.        	
    \end{proof}

    In the following, we take the  horseshoe  $\Lambda_n\subseteq \Lambda$ and measures $\mu_n',\nu_n'\in \mathcal{M}_{inv}(\Lambda_n)$ from Claim~\ref{Claim:horseshoe}.
    
\paragraph{Entropy of the $g$-irregular set $I_g$.}
By item~\ref{item:irregular-HS} of Corollary~\ref{Cor:Thompson-HS}, item (a) above implies that 
$$h_{\rm top}(I_g\cap\Lambda_n)=h_{\rm top}(\Lambda_n).$$
Thus by item (b.2) above, one has   
$$h_{\rm top}(I_g)\geq h_{\rm top}(I_g\cap\Lambda_n)=h_{\rm top}(\Lambda_n)\geq h_{\nu_n'}(X)>h_{\mathcal{M}}(X) -\frac2n.$$ 
Let $n \to+\infty$ we conclude that Item~\ref{item:irregular-SH-Hclass} of Proposition~\ref{Pro:irregular}  holds.

\paragraph{Entropy of the level set $R_g(a)$.}  
        We claim there exists an invariant measure $\omega_n\in\mathcal{M}_{inv}(\Lambda_n)$ such that $\displaystyle \int g {\rm d}\omega_n=a$ and 
        $h_{\omega_n}(X)$ tends to $h_{\mathcal{M}}(a)$ as $n\to\infty$. 
    By item (b.1) above and the fact $\displaystyle \int g{\rm d}\mu_n=a$, one has that 
    $$\displaystyle a-\frac1n<\int g {\rm d}\mu_n'<a+\frac1n.$$
    Without loss of generality, we may assume that 
    $\displaystyle a\leq\int g {\rm d}\mu_n'<a+\frac1n$
    (the other case is analogous).
    By item (a) above, there exists $\theta_n\in\mathcal{M}_{inv}(\Lambda_n)$ such that 
    $$\displaystyle \underline a\leq \int g {\rm d}\theta_n<\frac{\underline a+a}{2}\quad \big(<a\big).$$
    Thus one has $$0<\frac{a-\underline a}{2}<\int g {\rm d}\mu_n'-\int g {\rm d}\theta_n<a-\underline a+\frac{1}{n}.$$
Then, by the affinity of the integral and the entropy function, one has 
that the probability measure $\displaystyle\omega_n=\frac{\int g {\rm d}\mu_n'-a}{\int g {\rm d}\mu_n'-\int g {\rm d}\theta_n}\theta_n+\frac{a-\int g {\rm d}\theta_n}{\int g {\rm d}\mu_n'-\int g {\rm d}\theta_n}\mu_n'$ satisfies 
    $\displaystyle\int g{\rm d}\omega_n=a,$
     and 
    \begin{align*}
    	h_{\omega_n}(X) 
    	&\geq \frac{a-\int g {\rm d}\theta_n}{\int g {\rm d}\mu_n'-\int g {\rm d}\theta_n}h_{\mu_n'}(X)\\
    	&=\left(1+\frac{a-\int g {\rm d}\mu'_n}{\int g {\rm d}\mu_n'-\int g {\rm d}\theta_n}\right)h_{\mu_n'}(X)\\
    	&\geq \left(1-\frac{1}{n}\frac{2}{a-\underline a}\right) \left(h_{ \mathcal{M}}(a) - \frac2n\right).
    \end{align*}
    By Item~\ref{item:levelset-HS} in Corollary~\ref{Cor:Thompson-HS}, one concludes that
    $$h_{\rm top}(R_g(a))\geq h_{\rm top}(R_g(a)\cap\Lambda_n)\geq h_{\omega_n}(X)\geq \left(1-\frac{1}{n}\frac{2}{a-\underline a}\right) \left(h_{ \mathcal{M}}(a) - \frac2n\right).$$
As the right hand-side above tends to $h_{ \mathcal{M}}(a)$ as $n$ tends to infinity, this proves
item~\ref{item:levelset-SH-Hclass} of Proposition~\ref{Pro:irregular}.
\end{proof}

\subsection{Proofs of Theorems~\ref{Thm:irregular-Lorenz} \&~\ref{Thm:irregular-SinHyp-attractor} }\label{sec:proofsAB}

We first prove  Theorem~\ref{Thm:irregular-levelset}.

\begin{proof}[Proof of Theorem~\ref{Thm:irregular-levelset}]
	Assume $X\in\mathscr{X}^1(M)$ and $\Lambda$ is a non-trivial singular hyperbolic homoclinic class of $X$ such that each pair of periodic orbits contained in $\Lambda$ are homocli\-nically related and $\mathcal{M}_{inv}(\Lambda)=\overline{\mathcal{M}_1(\Lambda)}$. Proposition \ref{Pro:Horseshoe-approximation-inv} implies $\mathcal{M}_{inv}(\Lambda)$ has the horseshoe approximation property.
	
	Let $g\in C(\Lambda,\mathbb{R})$. If there exists $a_0\in\mathbb{R}$ such that $\displaystyle \int g {\rm d}\mu=a_0$ for every $\mu\in\mathcal{M}_{inv}(\Lambda)$, then  $\Lambda=R_g(a_0)$ and hence $I_g=\emptyset$.
	
	Now assume  there are   $\omega_1,\omega_2\in \mathcal{M}_{inv}(\Lambda)$  such that  $\displaystyle\int g {\rm d}\omega_{1}\neq \int g {\rm d}\omega_{2}$.
	By Lemma~\ref{Lem:periodic measure dense}, there exist $\nu_1,\nu_2\in \mathcal{M}_{per}(\Lambda)$ such that
	$$
	\displaystyle\int g {\rm d}\nu_{1}\neq \int g {\rm d}\nu_{2}.
	$$
	Since each pair of periodic orbits contained in $\Lambda$ are homoclinic related, the stable manifold  of any periodic orbit is dense in $\Lambda$. Then~\cite[Theorem A]{CV21} implies that $I_{g}$ is residual  in $\Lambda$ 
(see alternatively the proof of~\cite[Corollary VI]{CV21}).

    \paragraph[$I_g$]{Entropy of $I_g$.}
    By Proposition~\ref{Pro:Horseshoe-entropy-invariant}, one has that $\mathcal{M}_{1}(\Lambda)$ satisfies the horseshoe approximation property, thus by Item~\ref{item:irregular-SH-Hclass} of Proposition~\ref{Pro:irregular}, one has  
    $$h_{\rm top}(I_{g})\geq h_{\mathcal{M}_1(\Lambda)}(X)\colon=\sup \{h_{\mu}(X)\colon \mu\in\mathcal{M}_1(\Lambda) \}.$$
    
    Recall that $\mathcal{M}_1(\Lambda)=\Big\{\mu\in \mathcal{M}_{inv}(\Lambda)\colon \mu(\sing(\Lambda))=0 \Big\}$ and $h_{\mu}(X)=0$ if $\mu(\sing(\Lambda))=1$. Thus by the variational  principle, one has 
    $$h_{\mathcal{M}_1(\Lambda)}(X)=\sup \{h_{\mu}(X)\colon \mu\in\mathcal{M}_{inv}(\Lambda) \}=h_{\rm top}(\Lambda).$$
    As a consequence,
    $$h_{\rm top}(I_{g})=h_{\rm top}(\Lambda).$$ 

	\paragraph{Denseness and entropy of $R_g(a)$.}
	Take $a\in\mathbb{R}$ such that 
	$$
	\displaystyle\inf_{\mu\in \mathcal{M}_1(\Lambda)}\int g {\rm d}\mu<a<\sup_{\mu\in \mathcal{M}_1(\Lambda)}\int g {\rm d}\mu.
	$$
	As above, this ensures that there exist $\mu_1,\mu_2\in \mathcal{M}_{per}(\Lambda)$ so that $\displaystyle\int g {\rm d}\mu_{1}<a< \int g {\rm d}\mu_{2}$.
	
	To obtain the denseness of $R_g(a)$ in $\Lambda$, one first constructs a nested sequence of horseshoes $\{\Lambda_n\}_{n\in\mathbb{N}}$  approximating $\Lambda$. 
	Note that all periodic orbits contained in $\Lambda$ are hyperbolic. 
	Since $\Lambda$ is non-trivial, there are countably infinitely many periodic orbits contained in $\Lambda$ and one lists them as $\{\gamma_n\}_{n\in\mathbb{N}}$. 
	Moreover, each pair of periodic orbits in $\Lambda$ are homoclinically related. 
	One constructs $\{\Lambda_n\}_{n\in\mathbb{N}}$ inductively as follows:
	\begin{itemize}
		\item Let $\Lambda_0$ be a horseshoe that contains $\gamma_0$ and $\gamma_1$. Such a horseshoe exists because  $\gamma_0$ and $\gamma_1$ are homoclinically related.
		\item For $n\geq 1$, let $\Lambda_n$ be a horseshoe that contains $\Lambda_{n-1}$ and $\gamma_n$. Such a horseshoe exists because  $\gamma_n$ is homoclinically related with $\Lambda_{n-1}$ in the sense that $\gamma_n$ is homoclinically with all periodic orbits contained in $\Lambda_{n-1}$. Note that if $\gamma_n$ is contained in $\Lambda_{n-1}$, then $\Lambda_n=\Lambda_{n-1}$.
	\end{itemize}
	By construction, one has that $\Lambda_n\subset\Lambda_{n+1}$ for all $n\in\mathbb{N}$. Moreover, by denseness of the periodic orbits in $\Lambda$
	one has that $\Lambda_n$ tends to $\Lambda$ (in the Hausdorff distance) as $n\rightarrow \infty$.
	Recall that one has assumed $\mu_1,\mu_2$ to be two periodic measures, thus there exists $n_0\in\mathbb{N}$ such that $\mu_1,\mu_2\in\mathcal{M}_{inv}(\Lambda_{n_0})$. 
	As a consequence, $\mu_1,\mu_2\in \mathcal{M}_{inv}(\Lambda_{n})$ for all $n\geq n_0$.
	By Item~\ref{item:generic-measure-HS} of Theorem~\ref{Thm:HS-Sigmund}, for each $n\geq n_0$, there exists $\mu_1^n,\mu_2^n\in \mathcal{M}_{inv}(\Lambda_{n})$ such that 
	$$
	\displaystyle\int g {\rm d}\mu_{1}^n<a< \int g {\rm d}\mu_{2}^n
	\qquad {\rm and } \qquad
	\supp(\mu_1^n)=\supp(\mu_2^n)=\Lambda_n.
	$$
	
	Take a suitable $\theta_n\in(0,1)$ for each $n\geq n_0$, such that the linear combination 
	$$
	\nu_n=\theta_n\mu_1^n+(1-\theta_n)\mu_2^n
	\qquad {\rm satisfies} \qquad
	\displaystyle\int g {\rm d}\nu_n=a.
	$$ 
	Note that $\supp(\nu_n)=\Lambda_n$. By Item~\ref{item:mu-generic-HS} of Theorem~\ref{Thm:HS-Sigmund}, the set $G_{\nu_n}$ of $\nu_n$-generic points is non-empty. 
	Take $x_n\in G_{\nu_n}\cap\Lambda_n$, then $\supp(\nu_n)\subset\omega(x_n,\Phi)$ where $\omega(x_n,\Phi)$ is the positive limit set of $x_n$. This implies that $\omega(x_n,\Phi)=\Lambda_n$. Note that $\orb(x_n)\subset G_{\nu_n}$, hence $G_{\nu_n}$ is dense in $\Lambda_n$. As a consequence, one has $\bigcup_{n\geq n_0}G_{\nu_n}$ is dense in $\Lambda$. By the fact that $\bigcup_{n\geq n_0}G_{\nu_n}\subset R_g(a)$, one has $R_g(a)$ is dense in $\Lambda$.
	\medskip
	

	The estimation of the entropy $h_{\rm top}(R_g(a))$ follows similar arguments as for $h_{\rm top}(I_g)$. Since $\mathcal{M}_{inv}(\Lambda)$ satisfies the horseshoe approximation property, using Item~\ref{item:levelset-SH-Hclass} of Proposition~\ref{Pro:irregular}, one has  
	\begin{align*}
	  h_{\rm top}(R_g(a))\geq h_{\mathcal{M}_{inv}(\Lambda)} (a)\colon
	  =\sup\left\{h_\mu(X)\colon \mu\in   \mathcal{M}_{inv}(\Lambda) \text{~and ~}\int g {\rm d}\mu=a \right\}.
	\end{align*}
%
	The inverse inequality is obtained as an adaptation of Bowen's result~\cite[Theorem 2]{Bowen-entropy} to the flow case, as we now explain. 
	For a point $x\in\Lambda$, denote by $V(x)$ and $V(x,\phi_1)$ the limit sets of empirical measures of $x$ under the action of the flow $(\phi_t)_t$ and its time-one map $\phi_1$, respectively. In other words,
	$$V(x)=\left\{\mu\in\mathcal{M}_{inv}(\Lambda)\colon \exists t_i\rightarrow+\infty \text{~s.t.~} \mu=\lim_{i\rightarrow\infty}\frac{1}{t_i}\int_0^{t_i} \delta_{\phi_s(x)} \, {\rm d}s\right\}
	$$
	and
	$$V(x,\phi_1)=\left\{\mu\in\mathcal{M}_{inv}(\Lambda,\phi_1)\colon \exists n_i\rightarrow+\infty \text{~s.t.~} \mu=\lim_{i\rightarrow\infty}\frac{1}{n_i}\sum_{k=0}^{n_i} \delta_{\phi_k(x)} \, \right\}.$$
	To simplify notations, let $\displaystyle b=\sup\Big\{h_\mu(X)\colon \mu\in   \mathcal{M}_{inv}(\Lambda) \text{~and ~}\int g {\rm d}\mu=a \Big\}$. We need the following:

\medskip
\begin{claim*}
For any $x\in R_g(a)$ and any $\nu\in V(x,\phi_1)$, one has $h_{\nu}(\phi_1)\leq b$.
\end{claim*}

    \begin{proof}[Proof of the claim]
    	Take $x\in R_g(a)$. Each $\mu\in V(x)$ satisfies that $ \int g {\rm d}\mu=a$ and thus $h_\mu(X)\leq b$. On the other hand, for any $\nu\in V(x,\phi_1)$,  the measure $\mu=\int_0^1 (\phi_s)_*\nu \, {\rm d}s$ is invariant by the flow and belongs to $V(x)$. As a consequence, one has $h_{\nu}(\phi_1)=h_{\mu}(\phi_1)=h_{\mu}(X)\leq b$ since the metric entropy is affine on the space of invariant probability measures. 
    \end{proof}
    The above claim implies that
    $$R_g(a)\subset QR(b)\footnote{The notation $QR(b)$ is borrowed from~\cite[Theorem 2]{Bowen-entropy} which ensures that $h_{\rm top}(QR(b),\phi_1)\leq b$.}\colon=\{x\in\Lambda\colon \exists \nu\in V(x,\phi_1) \text{~\rm with~} h_{\nu}(\phi_1)\leq b\}.$$
	By~\cite[Theorem 2]{Bowen-entropy}, one concludes that 
	$$
	h_{\rm top}(R_g(a)) 
	\leq h_{\rm top}(QR(b),\phi_1)\leq b.
	$$


	\bigskip
	Finally, it remains to show that the set 
	$$\hat C=\left\{g\in C(\Lambda,\mathbb{R})\colon \int g {\rm d}\mu_{1}\neq \int g {\rm d}\mu_{2} \text{~for some~} \mu_1,\mu_2\in \mathcal{M}_1(\Lambda)\right\}$$
	is open and dense in $C(\Lambda,\mathbb{R})$.
	The openness is by continuity of the integrals in the weak$^*$ topology.  
	To prove denseness, take two different periodic orbits $\orb(p),\orb(q)$ in $\Lambda$, take $\hat g\in C(\Lambda,\mathbb{R})$ such that 
	$$\hat g|_{\orb(p)}=0,~\hat g|_{\orb(q)}=1~\text{~and~} 0\leq\hat g(x)\leq 1~\forall x\in\Lambda.$$
	Such $\hat g$ exists since  $\orb(p)$ and $\orb(q)$ are two distinct periodic orbits.
	Let $\nu_1,\nu_2$ be the two periodic measures associated to $\orb(p),\orb(q)$ respectively.
	Note that $\nu_1,\nu_2\in\mathcal{M}_1(\Lambda)$.
	For any $g\in C(\Lambda,\mathbb{R})$, 
	\begin{itemize}
		\item if $\displaystyle \int g {\rm d}\nu_{1}\neq \int g {\rm d}\nu_{2}$, then $g\in\hat C$;
		\item if otherwise $\displaystyle \int g {\rm d}\nu_{1}= \int g {\rm d}\nu_{2}$, let $g_n=g+\frac{1}{n}\hat g$, then $g_n\in\hat C$ and $g=\lim\limits_{n\rightarrow \infty} g_n$.
	\end{itemize}
    This shows $\hat C$ is dense in $C(\Lambda,\mathbb{R})$ and completes the proof of Theorem~\ref{Thm:irregular-levelset}. 	 
\end{proof}

Now we are ready to prove Theorem~\ref{Thm:irregular-Lorenz} \&~\ref{Thm:irregular-SinHyp-attractor}.

\begin{proof}[Proofs of Theorem~\ref{Thm:irregular-Lorenz} \&~\ref{Thm:irregular-SinHyp-attractor}]
	By Item~\ref{item:generic-Lorenz-attractor} of Corollary~\ref{Pro:HS-approximation-generic}, there exists a residual subset $\mathcal{R}^r\subset\mathcal{X}^r(M^3)$ where $r\in\mathbb{N}_{\geq 2}$ such that for any $X\in\mathcal{R}^r$, if $\Lambda$ is a geometric Lorenz attractor for $X$, then $\mathcal{M}_{inv}(\Lambda)=\overline{\mathcal{M}_1(\Lambda)}$. Moreover, by Proposition~\ref{Pro:lorenz-robust}, every pair of periodic orbits are homoclinically related. 
		
	By Item~\ref{item:generic-SH-attractor} of Corollary~\ref{Pro:HS-approximation-generic}, there exists a residual subset $\mathcal{R}\subset\mathcal{X}^1(M)$ such that for any $X\in\mathcal{R}$, if $\Lambda$ is a singular hyperbolic attractor for $X$, then $\mathcal{M}_{inv}(\Lambda)=\overline{\mathcal{M}_1(\Lambda)}$ and 
	every pair of periodic orbits are homoclinically related by~\cite[Theorem B]{CY-Robust-attractor}. 
    
    Then Theorem~\ref{Thm:irregular-Lorenz} \&~\ref{Thm:irregular-SinHyp-attractor} are direct consequences of Theorem~\ref{Thm:irregular-levelset}.
\end{proof}

\section{Large deviations}\label{Section:proof-LD}

Here we will focus on large deviations for singular hyperbolic attractors, including the geometric Lorenz attractor. The theory of large deviations for singular hyperbolic attractors is still not completely understood, appart from the level-1 large deviations upper bounds associated to its SRB measure in  \cite{Ar07,AST19,CYZ}. 
This section is devoted to the proof of Theorem~\ref{Thm:LD2B-Lorenz-SH-attractor}, which generalizes the large deviation results by 
L. S. Young~\cite{You90} for flows with singularities. 
Our approach is inspired by \cite[Section~3]{PS05}, which establishes criteria for level-2 large deviations principles
for discrete-time dynamical systems. 
We overcome this fact dealing simultaneously with the flow dynamics (using the horseshoe approximation property and corresponding entropy denseness results) and induced discrete-time dynamics (taking suitable time-{$s_0$} maps). 
Due to the presence of singularities, we still have to consider the following special subset 
$$
\mathcal{M}_1(\Lambda)=\Big\{\mu\in \mathcal{M}_{inv}(\Lambda)\colon \mu(\sing(\Lambda))=0 \Big\}
$$
of $\mathcal{M}_{inv}(\Lambda)$. Recall that $\mathcal{M}(\Lambda)$ is the space of all probability measures supported on $\Lambda$.
We prove the following theorem in this section.
\begin{theorem}\label{Thm:LD2B} (Level-2 large deviations) 
    Let $X\in\mathscr{X}^1(M)$ and $\Lambda$ be a singular hyperbolic homoclinic class such that each pair of periodic orbits in $\Lambda$ are homoclinically related and $\overline{\mathcal M_{1}(\Lambda)}=\mathcal M_{inv}(\Lambda)$.   
    Assume $\mu_\psi$ is a weak Gibbs measure with respect to a H\"older continuous potential $\psi\colon \Lambda\to\mathbb R$ with $\Lambda_H$ being the $\mu_\psi$-full measure set such that \eqref{Def:Gibbs-weak} satisfies. 
    Then one has:
    \begin{enumerate}
    	\item\label{item:upper} (upper bound) There exists $c_\infty\leq 0$ so that 
    	\begin{align*}
    		\limsup_{t\to\infty} & \frac1t \log \mu_\psi \big(\{x\in \Lambda \colon \cE_t(x) \in \mathcal K\} \big)  
    		\le 
    		\max\Big\{ 
    		-\inf_{\mu \in \cK} \mathfrak I_{\psi}(\mu) \;,\;  c_\infty
    		\Big\}
    	\end{align*}
    	for any closed subset $\cK\subset \mathcal M(\Lambda)$.
    	
    	\item\label{item:lower} (lower bound)    If $\cO \subset \mathcal M(\Lambda)$ is an open set and $\nu\in \cO$ is ergodic satisfying $\nu(\Lambda_H)=1$, then
    	\begin{align*}
    		\liminf_{t\to +\infty} \frac1t \log \mu_\psi \Big( \Big\{x\in \Lambda \colon \cE_t(x)\in \cO \Big\} \Big)
    		\ge  \displaystyle -P_{\rm top}(\Lambda,\psi) + h_\nu(X) + \int {\psi} \, {\rm d}\nu.
    	\end{align*}
    	
    	\item\label{item:lower-Gibbs} (lower bound for Gibbs measure) If $\mu_\psi$ is a Gibbs measure with respect to $\psi$, then
    	\begin{align*}
    		\liminf_{t\to +\infty} \frac1t \log \mu_\psi \Big( \Big\{x\in \Lambda \colon \cE_t(x)\in \cO\Big\} \Big)
    		\ge  -\inf_{\mu \in \cO} \mathfrak I_{\psi}(\mu)
    	\end{align*}
    	for any open subset $\cO \subset \mathcal M(\Lambda)$.
    	
    \end{enumerate}
    
\end{theorem}

For completeness of the paper, we give a short explanation of the proof of Theorem~\ref{Thm:LD2B-Lorenz-SH-attractor}.
\begin{proof}[Proof of Theorem~\ref{Thm:LD2B-Lorenz-SH-attractor}]
	One takes the residual subset $\mathcal{R}^r\subset\mathcal{X}^r(M^3)$ where $r\in\mathbb{N}_{\geq 2}$  and the residual subset $\mathcal{R}\subset\mathcal{X}^1(M)$ in the proofs of Theorem~\ref{Thm:irregular-Lorenz} \&~\ref{Thm:irregular-SinHyp-attractor}. 	
	In both cases when $\Lambda$ is a geometric Lorenz attractor for $X\in\mathcal{R}^r$ or $\Lambda$ is a singular hyperbolic attractor for $X\in\mathbb{R}$, one has that $\mathcal{M}_{inv}(\Lambda)=\overline{\mathcal{M}_1(\Lambda)}$ and $\Lambda$ is a singular hyperbolic homoclinic class in which every pair of periodic orbits are homoclinically related. 
	Thus Theorem~\ref{Thm:LD2B-Lorenz-SH-attractor} is a consequence of Theorem~\ref{Thm:LD2B}	
\end{proof}
\medskip

In what follows, unless emphasized, we assume that  $\Lambda$ is a singular hyperbolic homoclinic class of $X\in\mathscr{X}^1(M)$ such that each pair of periodic orbits in $\Lambda$ are homoclinically related and we also assume that $\overline{\mathcal M_{1}(\Lambda)}=\mathcal M_{inv}(\Lambda)$. 
Assume $\mu_{\psi}\in \mathcal{M}_{inv}(\Lambda)$ is a weak Gibbs measure with respect to a H\"older continuous potential $\psi \colon\Lambda\rightarrow \mathbb R$ and $\Lambda_H\subset\Lambda$ be the $\mu_\psi$-full measure set satisfied for~\eqref{Def:Gibbs-weak}.
To be precise, there exists $\eps_0>0$ such that for any $x\in \Lambda_H,t>0$ and $\varepsilon\in(0,\eps_0)$, there exist constants $C_t(x,\eps)>0$ satisfying: 
\begin{equation}\label{Eq:Gibbs-weak}
	\frac1{C_t(x,\eps)} \le \frac{\mu_\psi \Big( B\big(y,t,\eps,\Phi\big)  \Big)}{e^{- t\,P_{\rm top}(X,\psi) + \int_0^t \psi(\phi_s(x))\, {\rm d}s}}
	\le  C_t(x,\eps).
\end{equation}
for any dynamic Bowen ball $B(y,t,\eps,\Phi)\subset B(x,t,\eps_0,\Phi)$.

\subsection{Lower bound}  
We give the lower bounds of large deviations in this section, that is to prove item~\ref{item:lower} \&~\ref{item:lower-Gibbs} of Theorem~\ref{Thm:LD2B}. The following instrumental result proves item~\ref{item:lower}.
  
\begin{proposition}\label{prop:lower-wGibbs}
    Let $\Lambda, \psi, \mu_\psi$ and $\Lambda_H$  be as in the assumption above.
    If $\cO \subset \mathcal M(\Lambda)$ is an open set, $\nu\in \cO$ is ergodic {and $\nu(\Lambda_H)=1$} then
    \begin{align*}
    \liminf_{t\to +\infty} \frac1t \log \mu_\psi \Big( \Big\{x\in \Lambda \colon \cE_t(x)\in \cO \Big\} \Big)
    \ge  \displaystyle -P_{\rm top}(\Lambda,\psi) + h_\nu(X) + \int {\psi} \, {\rm d}\nu.
    \end{align*}
\end{proposition}

\begin{proof}
The argument is inspired by \cite[Proposition~3.1]{PS05}, with the necessary adaptations to the context of weak Gibbs measures.
Recall that one takes a dense subset $\big\{\varphi_i\big\}_{i=1}^{\infty}$ of $C(\Lambda,\mathbb{R})$ where  $\varphi_i$ is not the zero function for every $i\geq 1$ and, for any  $\mu,\nu\in \mathcal{M}(\Lambda)$,
$$\displaystyle d^*(\mu,\nu)=\sum_{i=1}^{\infty}\frac{|\int \varphi_i {\rm d}\mu-\int \varphi_i {\rm d}\nu|}{2^i\|\varphi_i\|}.$$
In consequence:
\begin{enumerate}
\item[(i)]
 $\displaystyle d^*(\beta\mu,\beta\nu)=\beta\, \displaystyle d^*(\mu,\nu)$ for every $\mu,\nu\in \mathcal{M}(\Lambda)$ and $\beta>0$,
\item[(ii)] $\displaystyle d^*(\mu_1+\mu_2,\nu_1+\nu_2) \le \displaystyle d^*(\mu_1,\nu_1)+\displaystyle d^*(\mu_2,\nu_2)$
for every $\mu_1,\mu_2,\nu_1,\nu_2\in \mathcal{M}(\Lambda)$.
 \end{enumerate}
Since $\cO$ is open in the weak$^*$-topology and $\nu\in \cO$, by the definition of the weak$^*$ topology 
one may choose $\delta>0$ and finitely many functions $\varphi_1, \dots, \varphi_{i_0} \in C(\Lambda,\mathbb R)$ as above so that 
the open neighborhood 
$$
\cO^{3\delta}:=\Big\{ \eta \in \mathcal M(\Lambda) \colon  \Big|\int \varphi_i\, d\nu -\int \varphi_i \, d\eta\Big|<3\delta, \; \forall \,1\le i \le i_0\Big\}
$$
is contained in $\cO$ and 
$$\displaystyle \left|\int \psi\, d\nu -\int \psi \, d\eta\right|<4\delta\quad \text{for any} \quad \eta\in\cO^{3\delta}.
$$
 
As $\mu_\psi$ is a weak Gibbs measure, recall that for any
$x\in \Lambda_H$, $\eps>0$ and $t>0$ there exist constants $C_t(x,\eps)>0$ satisfying ~\eqref{Eq:Gibbs-weak}. 
Let $\eps>0$ be small and fixed (to be chosen below).
Firstly, by \cite{PS}, the set of real numbers $s_0>0$ so that $\nu$ is ergodic for the time-$s_0$ map $f=\phi_{s_0}$ is Baire generic. One chooses  such an $s_0>0$ small enough such that
$$
\sup_{x\in \Lambda} d^*\Big(\frac1{s_0} \int_0^{s_0} \delta_{\phi_s(x)} \, ds, \delta_x \Big) <\delta
$$
and let $f=\phi_{s_0}$ denote the time-$s_0$ map.
Then Proposition~2.1 in \cite{PS05} applied to the open neighborhood $\cO^\delta$ of $\nu$ ensures that there is $N\ge1$ and for every $n\ge N$ there exists a set
\footnote{Equation ~\eqref{eq:mod} is a modified version of the statement of Proposition 2.1. Nevertheless, in the notation of \cite{PS05}, 
the argument carries out identically, 
if one replaces the sets $X_{n,F}^{B^\delta}$ at equation (2.25) in \cite{PS05} by $X_{n,F}^{B^\delta} \cap A_n$ for some family of sets $(A_n)_{n\ge 1}$ such that $\nu(A_n)$ tends to 1 as $n~\to\infty$.
This is because of the fact $\nu(\Lambda_H)=1$.} 
\begin{equation}\label{eq:mod}
D\subset \left\{ x\in \Lambda_H  \colon \mathcal E^f_n(x) \in \cO^\delta \; \&\; C_{s_0 n}(x,\eps/2) \le e^{\delta s_0 n}\right\} ~~~\text{where} ~~\mathcal E^f_n(x) =\frac{1}{n}\sum_{i=0}^{n-1}\delta_{f^i(x)}
\end{equation}
such that $D$ is $(n,\eps)$-separated and has cardinality larger than or equal to $e^{n (h_\nu(f)-\delta)}$.
Using that $f=\phi_{s_0}$ it is clear that the set $D$ is $(s_0 n, \eps)$-separated with respect to the flow $\Phi=(\phi_t)_t$. 
Taking $t=s_0 n$, properties (i) and (ii) above imply 
\begin{align*}
d^*(\mathcal E_t(x),\mathcal E^f_n(x))
	& = \frac1n d^*\Big(\sum_{j=0}^{n-1} \frac1{s_0} \int_0^{s_0} \delta_{\phi_s(f^j(x))} \, ds, \sum_{j=0}^{n-1} \delta_{f^j(x)} \Big) \\
	& \le \frac1n \sum_{j=0}^{n-1} d^*\Big( \frac1{s_0} \int_0^{s_0} \delta_{\phi_s(f^j(x))} \, ds, \delta_{f^j(x)} \Big)
	< \delta.
\end{align*}
Analogously, if $x\in \Lambda$ and $y\in B(x,t, \eps, \Phi)$ then 
\begin{align*}
d^*(\mathcal E_t(y),\mathcal E_t(x))
	& = \frac1n d^*\Big(\sum_{j=0}^{n-1} \frac1{s_0} \int_0^{s_0} \delta_{\phi_s(f^j(y))} \, ds, 
				\sum_{j=0}^{n-1} \frac1{s_0} \int_0^{s_0} \delta_{\phi_s(f^j(x))} \, ds \Big) \\
	& \le \frac1n \sum_{j=0}^{n-1} d^*\Big( \frac1{s_0} \int_0^{s_0} \delta_{\phi_s(f^j(x))} \, ds, \frac1{s_0} \int_0^{s_0} \delta_{\phi_s(f^j(y))} \, ds \Big)<\delta.
\end{align*}
Therefore one may reduce $\eps$, if necessary, to guarantee that the summands in the right hand side above are arbitrarily small and
consequently
$$
\bigcup_{x\in D} B(x,t, \eps, \Phi) \subset \left\{ x\in \Lambda \colon \mathcal E_{t}(x) \in \cO^{3\delta} \right\}.
$$
Therefore, using the definition of weak Gibbs measure in ~\eqref{Eq:Gibbs-weak}, 
that $h_\nu(f)=s_0 h_\nu(X)$ and that the dynamic balls $B(x,t, \eps/2, \Phi)$ are pairwise disjoint, one concludes that
\begin{align*}
\frac1t \log \mu_\psi \Big( \Big\{x\in \Lambda \colon \cE_t(x)\in \cO \Big\} \Big)
	& \ge \frac1t \log \sum_{x\in D} \mu_\psi \big(\, B(x,t, \eps/2, \Phi)\, \big)\\
	& \ge \frac1t \log \sum_{x\in D} \big[ C_t(x,\eps)^{-1} {e^{- t\,P_{\rm top}(X,\psi) + \int_0^t \psi(\phi_s(x))\, {\rm d}s}} \big]
\end{align*}
Taking $n\to\infty$, which makes $t\to\infty$ as well, and by the choice of $D$ in ~\eqref{eq:mod},  we conclude that 
$$
\liminf_{t\to +\infty} \frac1t \log \mu_\psi \Big( \Big\{x\in \Lambda \colon \cE_t(x)\in \cO \Big\} \Big)
    \ge  \displaystyle - P_{\rm top}(X,\psi) + h_\nu(X) + \int {\psi} \, {\rm d}\nu -\big(1+\frac1{s_0}\big)\delta.
$$
Since $\delta>0$ is small and arbitrary, this proves the proposition.
\end{proof}

As a consequence of Proposition~\ref{prop:lower-wGibbs}, one can now prove the lower bound estimate for Gibbs measures in item~\ref{item:lower-Gibbs} of Theorem~\ref{Thm:LD2B}.
More precisely:

\begin{corollary}\label{cor:lower}
Let $\Lambda$ and $\psi$ be from the assumption of Proposition~\ref{prop:lower-wGibbs}.
Assume further that $\mu_\psi$ is a Gibbs measure with respect to $\psi$.	
Given an open subset $\cO \subset \mathcal M(\Lambda)$ one has that 
\begin{align*}
\liminf_{t\to +\infty} \frac1t \log \mu_\psi \Big( \Big\{x\in \Lambda \colon \cE_t(x)\in \cO\Big\} \Big)
    \ge  -\inf_{\mu \in \cO} \mathfrak I_{\psi}(\mu).
\end{align*}
\end{corollary}

\begin{proof}
Take an open subset $\cO \subset \mathcal M(\Lambda)$, it is sufficient to show that, for each $\mu \in \cO$ one has
\begin{equation}\label{eq:aux1234}
\liminf_{t\to +\infty} \frac1t \log \mu_\psi \Big( \Big\{x\in \Lambda \colon \cE_t(x)\in \cO  \Big\} \Big)
	\ge  -\mathfrak I_{\psi}(\mu).
\end{equation}
Note that since $\mu_\psi$ is a Gibbs measure, the set $\Lambda_H$ can be chosen as $\Lambda$, thus one has $\nu(\Lambda_H)=\nu(\Lambda)=1$ for any $\nu\in\mathcal{M}(\Lambda)$.

If $\mu\notin \mathcal M_{inv}(\Lambda)$, then $\mathfrak I_{\psi}(\mu)=+\infty$ and there is nothing to prove. Hence one just needs to consider where 
 $\mu\in \mathcal M_{inv}(\Lambda)$. Since by assumption $\overline{\mathcal M_{1}(\Lambda)}=\mathcal M_{inv}(\Lambda)$, Proposition~\ref{Pro:Horseshoe-approximation-inv} guarantees that ${\mathcal M_{inv}(\Lambda)}$ admits the  horseshoe approximation property. Thus, for any $\varepsilon>0$ there exists 
$\nu_\varepsilon \in \mathcal{M}_{erg}(\Lambda)$ so that $\Lambda'=\supp(\nu_\varepsilon)$ is a horseshoe, and the followings are satisfied:
$$
	d^*(\nu_\varepsilon,\mu)<\frac{\varepsilon}{L} ~~~\text{and}~~~h_{\nu_\varepsilon}(X)>h_{\mu}(X)-\varepsilon,
$$
where $L=\|\psi\|$ is the supremum norm of $\psi$.
In particular one has $\displaystyle \left|\int \psi\, d\mu - \int \psi\, d\nu_\varepsilon\right|<\varepsilon$.
The following estimation holds
\begin{align*}
-\mathfrak I_\psi(\nu_\varepsilon) 
	&=  h_{\nu_\varepsilon}(X) + \int {\psi} \, {\rm d}\nu_\varepsilon  - P_{\rm top}(X,\psi) \\
	&>  h_\mu(X) + \int {\psi} \, {\rm d}\mu  - P_{\rm top}(X,\psi) -2\varepsilon \\
	&= -\mathfrak I_\psi(\mu) -2\varepsilon.
\end{align*}

Shrink $\varepsilon$ so that $\nu_\varepsilon\in \mathcal O$.
Note that since $\nu_\eps$ is ergodic and $\Lambda_H=\Lambda$, one applies Proposition~\ref{prop:lower-wGibbs} to $\cO$ and $\nu_\eps$ and obtains the following
\begin{equation*}
	\liminf_{t\to +\infty} \frac1t \log \mu_\psi \Big( \Big\{x\in \Lambda : \cE_t(x)\in \cO \Big\} \Big)
	\ge  -\mathfrak I_{\psi}(\nu_\varepsilon)> -\mathfrak I_\psi(\mu) -2\varepsilon. 
\end{equation*}
Corollary~\ref{cor:lower} is concluded since $\eps$ can be taken arbitrarily small.
\end{proof}

Proposition~\ref{prop:lower-wGibbs} together with Corollary~\ref{cor:lower} prove item~\ref{item:lower} \&~\ref{item:lower-Gibbs} of Theorem~\ref{Thm:LD2B}. 


\subsection{Upper bound}

The large deviations upper bounds for the flow 
are inspired by \cite[Theorem~A]{Va12}  and \cite[Theorem 3.2]{PS05}. A first fundamental step is the following instrumental result, which explores ideas from Misiurewicz's proof of the variational principle.


\begin{lemma}\label{le:auxPS2}
Let $\Lambda$ be an invariant compact set of a vector field $X\in\mathcal{X}^1(M)$ and
let  $D\subset \mathcal M(\Lambda)$ be a non-empty set.
If $s_D(t,\eps)$ denotes the maximal cardinality of $(t,\eps)$-separated sets
in 
$
\Big\{ x\in M \colon \cE_t(x) \in  D \Big\}
$
then, 
$$
\limsup_{t\to\infty} \frac1t \log s_D(t,\eps) \le \sup_{\eta \in \overline{D}^{co} \cap\; \mathcal M_{inv}(\Lambda)} \, h_\eta(X), \quad \text{for every $\eps>0$},
$$
where $\overline{D}^{co}$ denotes the closed convex hull of $D$.\\
If, in addition, the entropy function $\mathcal M_{inv}(\Lambda)\ni \mu \mapsto h_\mu(X)$
is upper semicontinuous, then 
$$
\limsup_{t\to\infty} \frac1t \log s_D(t,\eps)  \le \sup_{\eta \in \overline{D} \cap\; \mathcal M_{inv}(\Lambda)} \, h_\eta(X), \quad \text{for every $\eps>0$},
$$
where $\overline{D}$ denotes the closure of $D$.
\end{lemma}

\begin{proof}
The proof is inspired by \cite[Lemma 3.1]{PS05}, in the discrete time context. For completeness, we shall 
provide a sketch of the proof. 
Let $D\subset \mathcal M(\Lambda)$ be a non-empty set and let $\eps>0$. For each $t>0$ let $E_{t}\subset \{x\in \Lambda\colon \cE_t(x) \in D\}$ be a $(t,\eps)$-separated set (with respect to the flow $\Phi=(\phi_t)_t$) with cardinality $s_D(t,\eps)$. Choose a sequence $(t_n)_{n\ge 1}$ so that 
\begin{equation}\label{eq:disccont}
\limsup_{t\to\infty} \frac1t \log s_D(t,\eps) = \limsup_{n\to\infty} \frac1{t_n}\log s_D(t_n,\eps),
\end{equation}
and consider the probability measures
$$
\sigma_n:= \frac1{s_D(t_n,\eps)} \sum_{x\in E_{t_n}} \delta_x
\quad\text{and}\quad
\mu_n:= \frac1{s_D(t_n,\eps)} \sum_{x\in E_{t_n}} \cE_{t_n}(x).
$$
Up to consider some convergent subsequence, we may assume without loss of generality that $(\mu_n)_{n\ge1}$ is weak$^*$ convergent to $\mu$.  It is clear that 
$\mu\in \mathcal M_{inv}(\Lambda)$. Moreover, as the sequence $(\mu_n)_{n\ge 1}$ is a convex combination of probability measures in $D$ then 
$\mu\in \overline{D}^{co}$. Therefore, using \eqref{eq:disccont}, in order to prove the first statement in the lemma it is enough to show that 
\begin{equation}\label{eq:suffc}
\limsup_{n\to\infty} \frac1{t_n}\log s_D(t_n,\eps) \le h_\mu(X).
\end{equation}

By Gronwall's inequality, there exists $C>0$ 
so that 
$$
C^{-1} e^{-s\, \|X\|_\infty} \, d(x,y) \le d(\phi_s(x),\phi_s(y)) \le C e^{s\, \|X\|_\infty} \, d(x,y)
$$
for every $x,y\in \Lambda$ and $s\in[0,1]$. 
Since $E_{t_n}$ is  a $(t_n,\eps)$-separated set with respect to the flow $\Phi=(\phi_t)_t$,  for any $x,y\in E_{t_n}$, there exists $t\in[0,t_n]$ such that $d(\phi_t(x),\phi_t(y))>\varepsilon$. 
Thus by the fact that $t-\lfloor t\rfloor \in[0,1]$, one has 
$$d(\phi_{\lfloor t \rfloor}(x),\phi_{\lfloor t \rfloor}(y))\geq C^{-1}e^{-\|X\|_{\infty}} d(\phi_t(x),\phi_t(y))>\gamma$$   
where $\gamma=C^{-1}e^{-\|X\|_{\infty}}\varepsilon$.
Note that $\lfloor t \rfloor \in[0,\lfloor t_n \rfloor]$, so $E_{t_n}$ is a a $(\lfloor t_n \rfloor,\gamma)$-separated set with respect to the time-one map $\phi_1$.
Choosing a partition $\mathcal P$ of $\Lambda$ with $\diam(\cP)<\gamma$ and $\mu(\partial \cP)=0$, 
one concludes that each element of the partition $\bigvee_{j=0}^{\lfloor t_n\rfloor -1} \phi_{-j}(\cP)$ contains at most one  element of $E_{t_n}$,
and so
$$
H_{\sigma_n}\Big(\bigvee_{j=0}^{\lfloor t_n\rfloor -1} \phi_{-j}(\cP)\Big) = \log \# E_{t_n} = \log {s_D(t_n,\eps)}.
$$
It is not hard to check that the probability measures
$
\hat \mu_n:= \frac1{s_D(t_n,\eps)} \sum_{x\in E_{t_n}} \cE_{\lfloor t_n\rfloor}(x)
$
converge to $\mu$ as well. Moreover,
the argument used in the proof of the variational principle (see e.g. \cite[Lemma 3.1]{PS05} or \cite[Theorem~8.6]{Wa}) ensures that 
$$
\limsup_{n\to\infty} \frac 1{t_n} \log  {s_D(t_n,\eps)} 
	= \limsup_{n\to\infty} \frac 1{\lfloor t_n\rfloor} \log  {s_D(t_n,\eps)} \le h_{\mu}(\phi_1,\cP) \le h_{\mu}(\phi_1)=h_{\mu}(X).
$$
This proves ~\eqref{eq:suffc}, and the first statement in the lemma.

Now, assume that $D\subset \mathcal M(\Lambda)$ is a non-empty set and that $\mathcal M_{inv}(\Lambda)\ni \mu \mapsto h_\mu(X)$
is upper semicontinuous. As $\overline D$ is compact, for each $\delta>0$ there exists a finite open cover $\{B(\eta_i,\delta)\}$ of $D$ by balls of radius $\delta$. 
In particular there exists $\eta_{i_\delta}\in \mathcal M(\Lambda)$ so that
$$
\limsup_{t\to\infty} \frac1t \log s_D(t,\eps)
= \limsup_{t\to\infty} \frac1t \log s_{D\cap B(\eta_{i_\delta},\delta)}(t,\eps).
$$
Using the first statement of the lemma, there exists $\mu_\delta \in \mathcal M_{inv}(\Lambda)\cap \overline{B(\eta_{i_\delta},\delta)}$ so that 
$$
\limsup_{t\to\infty} \frac1t \log s_D(t,\eps) \le h_{\mu_\delta}(X).
$$
In particular, any weak$^*$ accumulation point $\mu$ of $(\mu_\delta)_{\delta>0}$ belongs to $\overline D$ and, by upper semicontinuity of the entropy, 
$
\limsup_{t\to\infty} \frac1t \log s_D(t,\eps) \le h_{\mu}(X).
$
This completes the proof of the lemma.
\end{proof}

The previous result will allow to obtain the desired large deviations upper bounds. We observe that $h_\mu(X)=h_\mu(f)$ where $f=\phi_1$ denotes the time-1
of the flow $\Phi=(\phi_t)_t$. Moreover, notice that the entropy function associated to singular-hyperbolic attractors is upper-semicontinuous (cf. \cite{PYY}).
Thus, the large deviations upper bound in item~\ref{item:upper} of Theorem~\ref{Thm:LD2B} is now a direct consequence of the following proposition.

\begin{proposition}\label{prop:upper-wGibbs}
Let $\Lambda$ be an invariant compact set of a vector field $X\in\mathcal{X}^1(M)$ and
let $\cK\subset \mathcal M(\Lambda)$ be a closed and convex subset so that $\mathcal K\cap \mathcal M_{inv}(\Lambda)\neq\emptyset$. Assume $\mu_{\psi}\in \mathcal{M}_{inv}(\Lambda)$ is a 
weak Gibbs measure with respect to a H\"older continuous potential $\psi \colon\Lambda\rightarrow \mathbb R$ and $\Lambda_H\subset\Lambda$ be the $\mu_\psi$-full measure set satisfied for~\eqref{Def:Gibbs-weak}. 
Consider the non-positive real number
\begin{equation}\label{cinfty}
c_\infty:=\limsup_{\delta\to 0} \; \limsup_{t\to\infty} \frac1t\log \mu_\psi \Big(\big\{x\in \Lambda \colon C_t(x,\eps) > e^{\delta t}\big\}\Big).
\end{equation}
Then 
\begin{align}
\limsup_{t\to\infty} & \frac1t \log \mu_\psi \big(\{x\in \Lambda \colon \cE_t(x) \in \mathcal K\} \big)  \nonumber \\
& \le 
\max\Big\{ 
\sup_{\mu\in \mathcal K \cap \mathcal M_{inv}(\Lambda)} \Big( \displaystyle -P_{\rm top}(\Lambda,\psi) + h_\nu(X) + \int {\psi} \, {\rm d}\nu \Big) \; ,\;  
c_\infty
\Big\}. \label{PS-conclusion}
\end{align}
Furthermore, if the entropy function 
$\mathcal M_{inv}(\Lambda)\ni \mu\mapsto h_\mu(X)$ is upper-semicontinuous then ~\eqref{PS-conclusion} holds even if $\mathcal K$ is not convex.
\end{proposition}

\begin{proof}
Assume first that $\cK\subset \mathcal M(\Lambda)$ is a closed and convex 
subset.  As $\psi \colon\Lambda\rightarrow\mathbb R$ is H\"older continuous, then it is bounded and, 
given $\delta>0$, one can write $\mathcal K=\bigcup_{j=0}^{N_\delta} \mathcal K_j$ where
$$
\mathcal K_j=\Big\{ \eta\in \mathcal K \colon \int \psi\, d\eta \in [ -\|\psi\|_\infty+ j \delta, -\|\psi\|_\infty+ (j+1)\delta  ]  \Big\},
$$ 
$N_\delta=\lfloor \frac2\delta \|\psi\|_\infty\rfloor+1$. 
Note that some of the sets $\mathcal K_j$, which are  closed and convex, may be empty.
For each non-empty $\mathcal K_j$, either there exists $t_*>0$  so that $\mu_\psi(\{x\in \Lambda \colon \cE_t(x) \in \mathcal K_j \})=0$ for every $t>t_*$
or $\mathcal K_j \cap \mathcal M_{inv}(\Lambda) \neq\emptyset$.
For that reason we will assume, without loss of generality, that all $\mathcal K_j\neq\emptyset$ intersect the space of invariant probability measures. 
Now, for each $0\le j\le N_\delta$ so that $\mathcal K_j\neq\emptyset$, one has 
\begin{equation}\label{eq:decompKj}
\big\{x\in \Lambda \colon \cE_t(x) \in \mathcal K_j\big\} \subset 
	 \big\{x\in \Lambda \colon \cE_t(x) \in \mathcal K_j \, \& \, C_t(x,\eps) \le e^{\delta t} \big\} 
	 \cup 
	  \big\{x\in \Lambda \colon  C_t(x,\eps) > e^{\delta t} \big\}. 
\end{equation}
The maximal cardinality of a $(t,\eps)$-separated set in the first set in the right hand-side above
is bounded above by $s_{\mathcal K_j}(t,\eps)$ which, by Lemma~\ref{le:auxPS2},  satisfies
$$
\limsup_{t\to\infty} \frac1n \log s_{\mathcal K_j}(t,\eps)
	\le \sup_{\eta \in {\mathcal{K}_j} \cap\; \mathcal M_{inv}(X)} \, h_\eta(X), \quad \text{for every $\eps>0$}.
$$
Given $\eps>0$ small and fixed, pick a $(t,\eps)$-maximal separated set $E_{j,t}\subset \big\{x\in \Lambda \colon \cE_t(x) \in \mathcal K_j\big\}$. If $\mathcal K_j\neq\emptyset$ then, using the weak Gibbs property, one ensures that
\begin{align*}
\mu_\psi\Big(\big\{x\in \Lambda \colon \cE_t(x) \in \mathcal K_j \, \& \, C_t(x,\eps) \le e^{\delta t} \big\}\Big)
	& \le \sum_{x\in E_{j,t}}   \mu_\psi\Big(B\big(x,t,\eps,\Phi\big)\Big) \\
	& \le \sum_{x\in E_{j,t}}  e^{\delta t} \; e^{- t\,P_{\rm top}(X,\psi) + t  \int \psi \, d\mathcal E_t(x)} \\
	& \le s_{\mathcal K_j}(t,\eps)\; e^{\delta t} \; e^{- t\,P_{\rm top}(X,\psi) + t \sup_{\eta \in \mathcal K_j} \int \psi\, d\eta}\\
	& \le s_{\mathcal K_j}(t,\eps)\; e^{2\delta t} \; e^{- t\,P_{\rm top}(X,\psi) + t \sup_{\eta \in \mathcal K_j \cap \mathcal M_{inv}(X)} \int \psi\, d\eta}
\end{align*}
and, consequently,
\begin{align*}
\limsup_{t\to\infty} 
	 \frac1t \log \mu_\psi & \Big(\big\{x\in \Lambda \colon \cE_t(x) \in \mathcal K_j \, \& \, C_t(x,\eps) \le e^{\delta t} \big\}\Big) \\
				& \le \sup_{\eta \in {\mathcal{K}_j} \cap \mathcal M_{inv}(\Lambda)} \, \Big\{ -P_{\rm top}(X,\psi)+ h_\eta(X) +\int \psi \, d\eta\Big\} + 2\delta
\end{align*}
This, combined with ~\eqref{eq:decompKj}, ensures that 
$
\limsup_{t\to\infty} \frac1t \log \mu_\psi \big(\{x\in \Lambda \colon \cE_t(x) \in \mathcal K_j\} \big)
$
is bounded above by 
\begin{align*}
\max\Big\{ 
&\sup_{\eta \in \mathcal{K} \cap \mathcal M_{inv}(\Lambda)} \, \Big\{ -P_{\rm top}(X,\psi)+ h_\eta(X) +\int \psi \, d\eta\Big\} + 2\delta \; , \\
	&\limsup_{t\to\infty} \frac1t\log \mu_\psi \Big(\big\{x\in \Lambda \colon C_t(x,\eps) > e^{\delta t}\big\}\Big) 
 \Big\}
\end{align*}
for each small $\delta>0$. Taking the $\limsup$ as $\delta \to0$ in each of the terms in the previous expression we conclude that
\begin{align*}
\limsup_{t\to\infty} & \frac1t \log \mu_\psi \big(\{x\in \Lambda \colon \cE_t(x) \in \mathcal K\} \big)  \\
& \le 
\max\Big\{ 
\sup_{\mu\in \mathcal K \cap \mathcal M_{inv}(\Lambda)} \Big( \displaystyle -P_{\rm top}(\Lambda,\psi) + h_\nu(X) + \int {\psi} \, {\rm d}\nu \Big), 
c_\infty
\Big\},
\end{align*}
thus proving the first statement in the proposition. Finally, if $\mathcal M_{inv}(M)\ni \mu \mapsto h_\mu(\phi_1)$
is upper semicontinuous then the large deviations upper bound holds for arbitrary closed sets $\mathcal{K}$ as a consequence of the previous argument and the corresponding statement in Lemma~\ref{le:auxPS2} for such class of sets. This finishes the proof of the proposition.
\end{proof}

\subsection{A local level-1 large deviations principle}

Finally we note that Theorem~\ref{Thm:LD2B} implies  large deviations principle for singular hyperbolic sets and averages of continuous observables. 
Indeed,  Theorem~\ref{Thm:LD2B} together with Corollary~\ref{cor:lower} and the contraction principle (see \cite{CR-L11}) implies on the following:

\begin{corollary}\label{Cor:LD1B}(Level-1 large deviations) 
Let $X\in\mathscr{X}^1(M)$ and $\Lambda$ be a singular hyperbolic homoclinic class such that each pair of periodic orbits in $\Lambda$ are homoclinically related and $\overline{\mathcal M_{1}(\Lambda)}=\mathcal M_{inv}(\Lambda)$.   
Assume $\mu_\psi$ is a Gibbs measure with respect
to a H\"older continuous potential $\psi\colon \Lambda\to\mathbb R$. 
For any continuous observable $g\colon \Lambda\to\mathbb R$
it holds that
\begin{align*}
\limsup_{t\to +\infty} \frac1t \log \mu_\psi \Big( \Big\{x\in \Lambda \colon \frac{1}{t}\int_{0}^{t} g(\phi_s(x)) \, {\rm d}s \in [a , b] \Big\}\Big)
    \le -\inf_{s \in [a , b]} I_{\psi,g}(s) 
\end{align*}
and
\begin{align*}
\liminf_{t\to +\infty} \frac1t \log \mu_\psi \Big( \Big\{x\in \Lambda \colon \frac{1}{t}\int_{0}^{t} g(\phi_s(x)) \, {\rm d}s \in (a , b)\Big\} \Big)
    \ge -\inf_{s \in (a , b)} I_{\psi,g}(s)
\end{align*}
where the lower-semicontinuous rate function $I_{\psi,g}$ is given by
$$
I_{\psi,g}(s)=\sup\Big\{ P_{\rm top}(\Lambda,\psi) - h_\eta(X) - \int {\psi} \, {\rm d}\eta \colon  \eta \in \mathcal M_{inv}(\Lambda), \; 
\int g\, {\rm d}\eta = s \Big\}.
$$ 
Moreover, if there exist $\mu_1,\mu_2\in \mathcal M_{inv}(\Lambda)$ so that $\displaystyle\int g {\rm d}\mu_1 \neq \int g {\rm d}\mu_2$ 
and $\displaystyle\int g \, {\rm d}\mu_{\psi}\notin [a , b]$ then the infima in the right-hand side of the previous inequalities are strictly negative.
\end{corollary}

\begin{remark*}
	Similarly as Theorem~\ref{Thm:irregular-Lorenz} \&~\ref{Thm:irregular-SinHyp-attractor} \&~\ref{Thm:LD2B-Lorenz-SH-attractor}, the conclusion of Corollary~\ref{Cor:LD1B} holds for Lorenz attractors of vector fields in a residual subset $\mathcal{R}^r\subset\mathcal{X}^r(M^3), ~(r\in\mathbb{N}_{\geq 2})$, and also holds for singular hyperbolic attractors of vector fields in a residual subset $\mathcal{R}\subset\mathcal{X}^1(M)$.
\end{remark*}

\subsection*{Acknowledgments}
 The authors would like to thank the anonymous referees for their valuable comments and suggestions. 
 The authors would also like to thank Professor Dejun Feng who pointed out to us the second item of Remark~\ref{Rem:Thm-A-B} and Remark~\ref{Rem:level-set}. 

\bibliographystyle{plain}

\flushleft{\bf Yi Shi} \\
School of Mathematics, Sichuan University, Chengdu, 610000, China\\
\textit{E-mail:} \texttt{shiyi@scu.edu.cn}\\

\flushleft{\bf Xueting Tian} \\
School of Mathematical Sciences, Fudan University, Shanghai, 200433,   China\\
\textit{E-mail:} \texttt{xuetingtian@fudan.edu.cn}\\

\flushleft{\bf Paulo Varandas} \\
Departamento de Matem\'atica, Universidade Federal da Bahia, Av. Ademar de Barros s/n, Salvador, 40170-110, Brazil\\
Centro de Matem\'atica da Universidade do Porto, Rua do Campo Alegre, Porto, Portugal
\textit{E-mail:} \texttt{paulo.varandas@ufba.br}\\

\flushleft{\bf Xiaodong Wang} \\
School of Mathematical Sciences,  CMA-Shanghai, Shanghai Jiao Tong University, Shanghai, 200240,  China\\
\textit{E-mail:} \texttt{xdwang1987@sjtu.edu.cn}\\

\end{document}